\documentclass[10pt]{amsart}

\usepackage[latin1]{inputenc}
\usepackage[english]{babel}
\usepackage{amsmath}
\usepackage{amsfonts}
\usepackage{amssymb}
\usepackage{amsthm}
\usepackage{epsfig}
\usepackage{graphicx,colortbl}
\usepackage{url,hyperref}
\usepackage{mathtools, color, changes}
\usepackage{tikz}

\makeatletter
\@namedef{subjclassname@2020}{%
  \textup{2020} Mathematics Subject Classification}
\makeatother

\newtheorem{thm}{Theorem}[section]
\newtheorem{lemma}[thm]{Lemma}
\newtheorem{cor}[thm]{Corollary}
\newtheorem{proposition}[thm]{Proposition}

\newtheorem{rmk}[thm]{Remark}

\def\hyp{\mathop\mathrm{hyp}\nolimits}

\def\R{\mathbb{R}}
\def\N{\mathbb{N}}
\def\Z{\mathbb{Z}}

\def\vol{\mathrm{vol}}

\def\2Z{2^{-m}\Z^n}

\numberwithin{equation}{section}

\begin{document}

\title{A discrete approach to Zhang's projection inequality}

\author{David Alonso-Guti\'errez}
\address{\'Area de an\'alisis matem\'atico, Departamento de matem\'aticas, Facultad de Ciencias, Universidad de Zaragoza, C/ Pedro cerbuna 12, 50009 Zaragoza (Spain), IUMA}
\email{alonsod@unizar.es}
\author{Eduardo Lucas}
\address{Departamento de Matemáticas, Universidad de Murcia, Campus de Espinardo, 30100 Murcia, Spain.}
\email{eduardo.lucas@um.es}
\author{Javier Mart\'in Go\~ni}
\address{\'Area de an\'alisis matem\'atico, Departamento de matem\'aticas, Facultad de Ciencias, Universidad de Zaragoza, C/ Pedro cerbuna 12, 50009 Zaragoza (Spain), IUMA and Faculty of Computer Science and Mathematics, University of Passau, Innstrasse 33, 94032 Passau, Germany}
\email{j.martin@unizar.es; javier.martingoni@uni-passau.de}

\thanks{The first author is supported by MICINN project PID2022-137294NB-I00 and DGA project E48\_23R. The second author is funded by ``Comunidad Aut\'onoma de la Regi\'on de Murcia a trav\'es de la convocatoria de Ayudas a proyectos para el desarrollo de investigaci\'on cient\'ifica y t\'ecnica por grupos competitivos, incluida en el Programa Regional de Fomento de la Investigaci\'on Cient\'ifica y T\'ecnica (Plan de Actuaci\'on 2022) de la Fundaci\'on S\'eneca-Agencia de Ciencia y Tecnolog\'ia de la Regi\'on de Murcia, Ref.21899/PI/22'' and the grant PID2021-124157NB-I00, funded by MCIN/AEI/10.13039/501100011033/``ERDF A way of making Europe". The third author is supported by the Austrian Science Fund (FWF) Project P32405 \textit{Asymptotic Geometric Analysis and Applications.}}

\subjclass[2020]{52C07,52A20}

\keywords{Zhang's inequality, Berwald's inequality, Ball bodies, covariogram function, lattice point enumerator.}

\begin{abstract}
In this paper we will provide a new proof of the fact that for any convex body $K\subseteq\R^n$
$$
\frac{{{2n}\choose{n}}}{n^n}n\int_0^\infty r^{n-1}\vol_n(K\cap(re_n+K))dr\leq\frac{(\vol_n(K))^{n+1}}{(\vol_{n-1}(P_{e_n^\perp}(K)))^n},
$$
where $(e_i)_{i=1}^n$ denotes the canonical orthonormal basis in $\R^n$, $P_{e_n^\perp}(K)$ denotes the orthogonal projection of $K$ onto the linear hyperplane orthogonal to $e_n$, and $\vol_k$ denotes the $k$-dimensional Lebesgue measure. This inequality was proved by Gardner and Zhang and it implies Zhang's inequality. We will use our new approach to this inequality in order to prove discrete analogues of this inequality and of an equivalent version of it, where we will consider the lattice point enumerator measure instead of the Lebesgue measure, and show that from such discrete analogues we can recover the aforementioned inequality and, therefore, Zhang's inequality.
\end{abstract}

\maketitle

\section{Introduction and notation}

Given a convex body $K\subseteq\R^n$, i.e., a compact convex set with non-empty interior, the quantity given by
$$
\Vert x\Vert=\Vert x\Vert_2\vol_{n-1}(P_{x^\perp}(K)),\quad x\in\R^n,
$$
defines a norm in $\R^n$ as a consequence of Cauchy's projection formula \cite[Equation (A.45)]{G}. Here $\Vert\cdot\Vert_2$ denotes the Euclidean norm, whose closed unit ball in $\R^n$ will be denoted by $B_2^n$, $\vol_k(A)$ denotes the $k$-dimensional volume (Lebesgue measure) of a set contained in a $k$-dimensional affine subspace, and $P_{x^\perp}(K)$ denotes the orthogonal projection of $K$ onto the linear hyperplane orthogonal to $x$.  The closed unit ball of this norm is a convex body, called the polar projection body of $K$, which will be denoted by $\Pi^*K$.

As a direct consequence of \cite[Theorem. 4.1.5]{G}, for any convex body $K\subseteq\R^n$, the quantity $(\vol_n(K))^{n-1}\vol_n(\Pi^*K)$ is an affine invariant, i.e., for any non-degenerate affine map $T$ we have that $$(\vol_n(T(K)))^{n-1}\vol_n(\Pi^*T(K))=(\vol_n(K))^{n-1}\vol_n(\Pi^*K).$$ It was proved in \cite{Pe} that among all $n$-dimensional convex bodies, the affinely invariant quantity $(\vol_n(K))^{n-1}\vol_n(\Pi^*K)$ is maximized if and only if $K$ is an ellipsoid, obtaining the following inequality, which is known as the Petty projection inequality and is stronger than the isoperimetric inequality:
$$
(\vol_n(K))^{n-1}\vol_n(\Pi^*K)\leq\left(\frac{\vol_n(B_2^n)}{\vol_{n-1}(B_2^{n-1})}\right)^n.
$$

In \cite{Zh}, Zhang proved a reverse Petty projection inequality, showing that among all $n$-dimensional convex bodies, the same affinely invariant quantity is minimized if and only if $K$ is a simplex (i.e., the convex hull of $n+1$ affinely independent points). Thus, for any convex body $K\subseteq\R^n$

\begin{equation}\label{eq:Zhang}
\frac{{{2n}\choose{n}}}{n^n}\leq(\vol_n(K))^{n-1}\vol_n(\Pi^*K).
\end{equation}

Several proofs and extensions of this inequality have been given in the last decades. See for instance \cite{AJV} and \cite{GZ}, where this inequality is shown to be extremely related to the covariogram function; \cite{ABG1}, where the inequality is extended to (and can be recovered from) the setting of log-concave functions; or \cite{LRZ}, where this inequality is extended to different measures other than the Lebesgue measure. Let us recall here that the covariogram function is defined as $g_K(x)=\vol_n(K\cap(x+K))$ and that it is supported on the difference body $K-K$, which is the Minkowski sum of $K$ and its reflection with respect to the origin, $-K$.

In this paper we will focus on the approach to this inequality given by the authors in \cite{GZ}, which was also considered in \cite{ABG2} to provide a different proof of Zhang's inequality in the setting of log-concave functions. Our motivation is to obtain an approach to Zhang's inequality in the discrete setting, where we will consider discrete measures instead of the Lebesgue measure. In this approach the authors defined for a convex body $K\subseteq\R^n$ and every $p>-1$, the radial $p$-th mean body of $K$, which we denote by $R_p(K)$. It is defined  by its radial function for every $p\in(-1,\infty)\setminus\{0\}$
\begin{equation}\label{eq:DefinitionRadialFuncion}
\rho_{R_p(K)}^p(\theta)=\frac{1}{\vol_n(K)}\int_K(\vol_1(K\cap\{x+\lambda\theta\,:\lambda\geq0\}))^pdx\quad\forall\theta\in S^{n-1}
\end{equation}
and if $p=0$ it is defined by $\displaystyle{\rho_0(\theta):=\lim_{p\to0}\rho_{R_p(K)}(\theta)}$. Let us recall at this point that a set $L$ with $0\in L$ is called a star set with center $0$ if for every $x\in L$ and every $\lambda\in[0,1]$, one has $\lambda x\in L$. The radial function of a star set $L$ with $0$ as a center is defined for every $\theta\in S^{n-1}$, the unit sphere in $\R^n$, as $\rho_L(\theta)=\sup\{r\geq0\,:\,r\theta\in L\}$. For any two compact star sets $L_1,L_2$ with $0$ as a center, $L_1\subseteq L_2$ if and only if $\rho_{L_1}(\theta)\leq\rho_{L_2}(\theta)$ for every $\theta\in S^{n-1}$. In particular, for any convex body $K\subseteq\R^n$ we have that $\rho_{\Pi^*K}(\theta)=\frac{1}{\vol_{n-1}(P_{\theta^\perp}(K))}$ for every $\theta\in S^{n-1}$.

The authors showed in \cite[Theorem 5.5]{GZ} that if $-1<p\leq q$ then
\begin{equation}\label{eq:Inclusionpq}
{{n+q}\choose{n}}^\frac{1}{q}R_q(K)\subseteq{{n+p}\choose{n}}^\frac{1}{p}R_p(K).
\end{equation}
Furthermore, the authors showed in \cite[Lemma 2.1, Lemma 3.1]{GZ} that the integral defining the radial function of $R_p(K)$ in \eqref{eq:DefinitionRadialFuncion} can also be written as an integral on $P_{\theta^\perp}(K)$ and as an integral on the interval $[0,\infty)$. More precisely, on the one hand, the authors proved in \cite[Lemma 2.1]{GZ} that for every $p>-1$ and every $\theta\in S^{n-1}$, if we denote by $\langle\theta\rangle$ the 1-dimensional linear subspace spanned by $\theta$, then
$$
\rho_{R_p(K)}^p(\theta)=\frac{1}{\vol_n(K)(p+1)}\int_{P_{\theta^\perp}(K)}(\vol_1(K\cap(y+\langle\theta\rangle)))^{p+1}dy.
$$
Therefore, when $p\to(-1)^+$ and $q=n$ we obtain the following inclusion relation
\begin{equation}\label{incl}
{{2n}\choose{n}}^\frac{1}{n}R_n(K)\subseteq n\vol_n(K)\Pi^*(K),
\end{equation}
where the right-hand side appears since, for any $\theta\in S^{n-1}$,
\begin{eqnarray*}
&&\lim_{p\to(-1)^+}{{n+p}\choose{n}}^\frac{1}{p}\rho_{R_p(K)}(\theta)\cr
&=&\lim_{p\to(-1)^+}\left(\frac{\Gamma\left(1+n+p\right)}{\Gamma\left(1+n\right)\Gamma\left(1+p\right)(p+1)\vol_n(K)}\int_{P_{\theta^\perp}(K)}(\vol_1(K\cap(y+\langle\theta\rangle)))^{p+1}dy\right)^\frac{1}{p}\cr
&=&\lim_{p\to(-1)^+}\left(\frac{\Gamma\left(1+n+p\right)}{\Gamma\left(1+n\right)\Gamma\left(2+p\right)\vol_n(K)}\int_{P_{\theta^\perp}(K)}(\vol_1(K\cap(y+\langle\theta\rangle)))^{p+1}dy\right)^\frac{1}{p}\cr
&=&\left(\frac{(n-1)!}{n!\vol_n(K)}\vol_{n-1}(P_{\theta^\perp}(K))\right)^{-1}=\frac{n\vol_n(K)}{\vol_{n-1}(P_{\theta^\perp}(K))}=n\vol_n(K)\rho_{\Pi^*(K)}.\cr
\end{eqnarray*}
On the other hand, the authors also showed in \cite[Lemma 3.1]{GZ} that if $p>0$ then
$$
\rho_{R_p(K)}^p(\theta)=\frac{p}{\vol_n(K)}\int_0^{\infty}r^{p-1}\vol_n(K\cap(r\theta+K))dr
$$
and then $R_p(K)$ coincides with the $p$-th Ball body of the covariogram function of $K$, which we will denote by $K_p(g_K)$, since this expression is precisely the one that defines the radial function of the $p$-th Ball body of the covariogram function of $K$ (see Section \ref{subsec:BallBodiesLogConcave}). Therefore, the inclusion relation \eqref{incl} reads
\begin{equation}\label{inclusion2}
\frac{{{2n}\choose{n}}}{n^n}n\int_0^\infty r^{n-1}\vol_n(K\cap(r\theta+K))dr\leq\frac{(\vol_n(K))^{n+1}}{(\vol_{n-1}(P_{\theta^\perp}(K)))^n}\quad\forall\theta\in S^{n-1},
\end{equation}
which is equivalent to the main inequality that we consider in this paper (see Theorem \ref{thm:ZhangPreIntegration} below). Besides, since $K_n(g_K)=R_n(K)$,
\begin{equation}\label{eq:RadialFunctioNthBallBodyCovariogram}
\rho_{K_n(g_K)}^n(\theta)=\frac{n}{\vol_n(K)}\int_0^\infty r^{n-1}\vol_n(K\cap(r\theta+K))dr\quad\forall\theta\in S^{n-1},
\end{equation}
and \eqref{incl} can also be written in terms of the $n$-th Ball body of the covariogram function as
\begin{equation}\label{inclusion}
{{2n}\choose{n}}^\frac{1}{n}K_n(g_K)\subseteq n\vol_n(K)\Pi^*(K).
\end{equation}

We would like to point out that if we denote by $(e_i)_{i=1}^n$ the canonical orthonormal basis in $\R^n$, by considering orthogonal transformations of a convex body, proving \eqref{inclusion2} for any convex body $K\subseteq\R^n$ and any $\theta\in S^{n-1}$ is equivalent to proving the following inequality, which we state as a theorem, for every convex body:
\begin{thm}\label{thm:ZhangPreIntegration}
Let $K\subseteq\R^n$ be a convex body. Then
$$
\frac{{{2n}\choose{n}}}{n^n}n\int_0^\infty r^{n-1}\vol_n(K\cap(re_n+K))dr\leq\frac{(\vol_n(K))^{n+1}}{(\vol_{n-1}(P_{e_n^\perp}(K)))^n}.
$$
\end{thm}
More precisely, given a convex body $K\subseteq\R^n$ and $\theta\in S^{n-1}$, applying Theorem \ref{thm:ZhangPreIntegration} to $U(K)$ for any orthogonal map $U$ such that $U^t(e_n)=\theta$ and taking into account that $\vol_n(U(K)\cap(re_n+U(K)))=\vol_n(K\cap(rU^t(e_n)+K))$ and that $P_{e_n^\perp}(U(K))=P_{(U^t(e_n))^\perp}(K)$ we obtain \eqref{inclusion2}. Integration in polar coordinates provides Zhang's inequality \eqref{eq:Zhang} (see Corollary \ref{cor:Zhang}).

Throughout the whole text $dG_k$ will denote the measure on $\R^k$ given by the lattice point enumerator, $G_k(A)=|A\cap\Z^k|$ for any Borel set $A\subseteq\R^k$, where we denote by $|\cdot|$ the cardinality of a set, and $dm_k$ will denote the Lebesgue measure on $\R^k$. Whenever $A\subseteq\R^n$ is contained in the affine subspace $x_0+\textrm{span}\{e_1,\dots,e_k\}$ for some $x_0\in\textrm{span}\{e_1,\dots,e_k\}^\perp$, we will denote $G_k(A)=G_n(A-x_0)$. We will denote by $d\mu$ the measure on $\R^n=\R^{n-1}\times\R$ given by $d\mu=d G_{n-1}\otimes dm_1$. That is, for every Borel set $A\subseteq\R^n$,
\begin{equation}\label{eq:DefinitionMu}
\mu(A)=\sum_{y\in e_n^\perp\cap \Z^n}\vol_1(A\cap(y+\langle e_n\rangle))=\sum_{y\in P_{e_n^\perp}(A)\cap\Z^n}\vol_1(A\cap(y+\langle e_n\rangle)),
\end{equation}
where the sum is understood as $0$ if $P_{e_n^\perp}(A)\cap\Z^n=\emptyset$. Such measure is constructed so that, when considering it in $\R^{n+1}$, i.e., $d\mu=dG_{n}\otimes dm_1$, then for any Borel set $A\subseteq\R^n$ we have that $G_n(A)$ coincides with the measure $\mu$ of the hypograph of $\chi_A$, the characteristic function of $A$:
$$
G_n(A)=\mu\left(\{(x,t)\in\R^{n+1}\,:\,0\leq t\leq\chi_A(x)\}\right).
$$
Moreover, for any Borel measurable $f:\R^n\to[0,\infty)$, we have that if $d\mu=dG_{n}\otimes dm_1$ is the measure $\mu$ considered in $\R^{n+1}$ we have that
$$
\int_{\R^n}f(x)dG_n(x)=\int_{\hyp(f)}d\mu(x,t),
$$
where $\hyp(f)$ is the hypograph of $f$
$$
\hyp(f):=\{(x,t)\in\R^{n+1}\,:\,0\leq t\leq f(x)\}.
$$
We will also denote by $C_k$ the set $C_k:=(-1,1)^k\times\{0\}^{n-k}\subseteq\R^n=\R^k\times\R^{n-k}$, which is a $k$-dimensional open (in the topology induced in $\R^k\times\{0\}^{n-k}$ by the standard topology in $\R^n$) cube. The group of orthogonal matrices of order $n$ will be denoted by $O(n)$ and the Steiner symmetrization of a bounded convex set $K$ with respect to the hyperplane $e_n^\perp$ (see Section \ref{subsec:SteinerSymmetrization}) will be denoted by $S_{e_n}(K)$.

Let us also point out that the authors proved in \cite[Theorem 5.5]{GZ} the inclusion relation \eqref{eq:Inclusionpq} by applying Berwald's inequality (see Theorem \ref{thm:Berwald} below) to the concave function $f_\theta(x)=\vol_1(K\cap\{x+\lambda\theta\,:\lambda\geq 0\})$. In order to obtain the inclusion relation \eqref{inclusion} or, equivalently \eqref{inclusion2} and then, Zhang's inequality, Berwald's inequality in the whole range of parameters $-1<p<q=n$ (and not just $0<p<q=n$) was needed.

The first aim of this paper is to give a different proof of Theorem \ref{thm:ZhangPreIntegration}. The main difference between this new proof and the one in \cite{GZ} is that we will only make use of Berwald's inequality in the range $0<p<q$. Since a version of Berwald's inequality in this range was proved in the discrete setting in \cite{ALY}, under the condition that the involved concave function attains its maximum at $0$, we will be able to use this approach in the discrete setting as well, obtaining the following theorem:

\begin{thm}\label{thm:DiscreteZhangPreIntegration}
Let $K\subseteq\R^n$ be a convex body satisfying $\displaystyle{\max_{y\in e_n^\perp}\vol_1(K\cap(y+\langle e_n\rangle))}=\vol_1(K\cap\langle e_n\rangle)$.  Then
$$
\frac{{{2n}\choose{n}}}{n^n}n\int_0^\infty r^{n-1}\mu(K\cap(re_n+K))dr\leq\frac{(\mu(S_{e_n}(K)+C_{n-1}))^{n+1}}{(G_{n-1}(P_{e_n^\perp}(K)))^n}.
$$
\end{thm}

\begin{rmk}
Notice that, even though the condition $\displaystyle{\max_{y\in e_n^\perp}\vol_1(K\cap(y+\langle e_n\rangle))}=\vol_1(K\cap\langle e_n\rangle)$ does not imply that $K\cap\Z^n\neq\emptyset$ (take, for instance, $K=\left(0,\frac{1}{2}\right)+\frac{1}{4}B_2^2\subseteq\R^2$), it implies that $0\in P_{e_n^\perp}(K)$. Thus, under this assumption, we have $G_{n-1}(P_{e_n^\perp}(K))\geq1$.

Besides, this condition means that the concave function $f:P_{e_n^\perp}(K)\to[0,\infty)$ given by $f(y)=\vol_1(K\cap(y+\langle e_n\rangle))$ attains its maximum at $0$, which is a condition that will be needed in order to apply the discrete version of Berwald's inequality (see Theorem \ref{thm:BerwaldDiscrete} below).
\end{rmk}

\begin{rmk}
Let us point out that, if $G_{n-1}(P_{e_n^\perp}(K))=1$, by means of Lemma \ref{lem:IdentitiesSum} below, we have that $P_{e_n^\perp}(K)\cap\Z^n=\{y_0\}$ for some $y_0\in\Z^n\cap e_n^\perp$ and
$$
n\int_0^\infty r^{n-1}\mu(K\cap(re_n+K))dr=\frac{1}{n+1}\vol_1(K\cap(y_0+\langle e_n\rangle))^{n+1}=\frac{(\mu(S_{e_n}(K)))^{n+1}}{(G_{n-1}(P_{e_n^\perp}(K)))^n}.
$$
In such case, Theorem \ref{thm:DiscreteZhangPreIntegration} does not provide a better estimate than this identity.
\end{rmk}

The measure $d\mu=d G_{n-1}\otimes dm_1$ is closely related to the measure $dG_n$ (see Lemma \ref{lem:muyGn} below). As a consequence of this relation we can obtain the following corollary, which gives a version of Theorem \ref{thm:ZhangPreIntegration} involving only the lattice point enumerator. This version still implies Theorem \ref{thm:ZhangPreIntegration}:

\begin{cor}\label{cor:ZhangPreIntegrationLatticePoint}
Let $K\subseteq\R^n$ be a convex body satisfying $\displaystyle{\max_{y\in e_n^\perp}\vol_1(K\cap(y+\langle e_n\rangle))}=\vol_1(K\cap\langle e_n\rangle)$. Then
\begin{eqnarray*}
&&\frac{{{2n}\choose{n}}}{n^n}n\int_0^\infty r^{n-1}G_n(K\cap(re_n+K))dr\leq\frac{{{2n}\choose{n}}}{n^n}\rho_{K-K}^n(e_n)G_{n-1}(P_{e_n^\perp}(K))\cr
&+&\frac{(G_n(S_{e_n}(K)+C_{n-1})+G_{n-1}(P_{e_n^\perp}(K)+C_{n-1}))^{n+1}}{(G_{n-1}(P_{e_n^\perp}(K)))^n}.
\end{eqnarray*}
\end{cor}

Theorem \ref{thm:ZhangPreIntegration} can also be written (see Lemma \ref{lem:IdentitiesIntegral} below) in the following way:
\begin{thm}\label{thm:ZhangPreIntegration2}
Let $K\subseteq\R^n$ be a convex body. Then
$$
\frac{{{2n}\choose{n}}}{n^n}2^n\int_{S_{e_n}(K)}|\langle x,e_n\rangle|^ndx\leq\frac{(\vol_n(K))^{n+1}}{(\vol_{n-1}(P_{e_n^\perp}(K)))^n}.
$$
\end{thm}
Let us point out that this theorem also follows from \cite[Theorem 3]{Fra} (see also \cite[Corollary 2.7]{MP} for a proof in the centrally symmetric case). We will consider a discrete version of Theorem \ref{thm:ZhangPreIntegration2} in which all the measures involved are given by the lattice point enumerator. We will prove the following theorem where, for any $m>0$ and $p\geq 1$, $B_{m}(p)$ denotes the number
\begin{equation}\label{eq:DefinitionBm(p)}
B_{m}(p)=\sum_{k=0}^{\lfloor m\rfloor}\frac{p}{m}\left(1-\frac{k}{m}\right)^{n-1}\left(\frac{k}{m}\right)^{p-1},
\end{equation}
convening that $0^0=1$.

\begin{thm}\label{thm:PurelyDiscreteZhang}
Let $K\subseteq\R^n$ be a convex body with $0\in P_{e_n^\perp}(K)$. Let us assume that \\$\displaystyle{\max_{y\in P_{e_n^\perp}(K)\cap\Z^n}G_1(S_{e_n}(K)\cap(y+\langle e_n\rangle))=G_1(S_{e_n}(K)\cap\langle e_n\rangle)}$. There exists \\$\displaystyle{m_0\geq M:=\max_{x\in S_{e_n}(K)\cap\Z^n}\langle x,e_n\rangle}$ such that $m_0>1$
\begin{eqnarray*}
&&\frac{(n+1)B_{m_0}(n+1)^{-1}}{B_{m_0}(1)^{-(n+1)}}2^n\int_{S_{e_n}(K)}|\langle x,e_n\rangle|^ndG_n(x)\leq\cr
&\leq&\frac{\left(G_n(S_{e_n}(K)+C_{n-1})+G_{n-1}(P_{e_n^\perp}(K)+C_{n-1})\right)^{n+1}}{(G_{n-1}(P_{e_n^\perp}(K)))^n}.
\end{eqnarray*}
\end{thm}

\begin{rmk}\label{rmk:Condition0InProjection}
Notice that the condition $0\in P_{e_n^\perp}(K)$ implies that $P_{e_n^\perp}(K)\cap\Z^n\neq\emptyset$ and $S_{e_n^\perp}(K)\cap\Z^n\neq\emptyset$. In fact, $P_{e_n^\perp}(K)\cap\Z^n\neq\emptyset$ if and only if $S_{e_n^\perp}(K)\cap\Z^n\neq\emptyset$. Notice that this condition also implies that $G_1(S_{e_n}(K)\cap\langle e_n\rangle)\geq1$.
\end{rmk}

\begin{rmk}\label{rmk:RemarkM=0}
Since the integral on the left-hand side is $0$ if $M=0$, in such case we can choose any $m_0>1$ and the left-hand side is well-defined and equal to $0$. Thus, the inequality in Theorem \ref{thm:PurelyDiscreteZhang} is trivial if $M=0$.
\end{rmk}

\begin{rmk}
It will be seen in  Corollary \ref{prop:FromPurelyDiscreteZhangDiscreteZhangToZhang} that Theorem \ref{thm:PurelyDiscreteZhang} also implies Theorem \ref{thm:ZhangPreIntegration2} which, by Lemma \ref{lem:IdentitiesIntegral}, is equivalent to Theorem \ref{thm:ZhangPreIntegration}.
\end{rmk}

Finally, motivated by the fact that Theorem \ref{thm:ZhangPreIntegration} provides an upper bound for the radial function of the $n$-th Ball body of the covariogram function $g_K$ of a convex body $K$ in the direction $e_n$, and the discrete version of it given by Corollary \ref{cor:ZhangPreIntegrationLatticePoint} provides an upper bound for the radial function of the $n$-th Ball body of the discrete covariogram function $\tilde{g}_K(x)=G_n(K\cap(x+K))$ of a convex body $K$ in the direction $e_n$, we initiate the study of the family of $p$-th Ball bodies of the discrete covariogram function.  We will show that, even though they might be non-convex, the convex hull of the $p$-th Ball body of $\tilde{g}_K$ is included in a homothetic copy of the $p$-th Ball body of $\tilde{g}_{K+C_n}$. We will also prove an inclusion relation similar to the one given by \eqref{eq:Inclusionpq}, whenever $0<p<q$.

The paper is organised as follows: In Section \ref{sec:Preliminaries} we will introduce some known preliminary results on which our proofs will rely. In Section \ref{sec:NewProof} we will provide our new proof of Theorem \ref{thm:ZhangPreIntegration}. In Section  \ref{sec:DiscreteZhang} we will prove Theorem \ref{thm:DiscreteZhangPreIntegration}. In order to do that, we will follow the approach used in this new proof of Theorem \ref{thm:ZhangPreIntegration}. In Section \ref{sec:AnotherDiscreteApproach}, we will prove Theorem \ref{thm:PurelyDiscreteZhang}. Finally, in Section \ref{sec:BallBodies} we will introduce the family of $p$-th Ball bodies of the discrete covariogram function and study their convexity and inclusion relations.

\section{Preliminaries}\label{sec:Preliminaries}

In this section we will introduce some well-known results that will be used in our proofs.

\subsection{The lattice point enumerator}\label{subsec:LatticePoint enumerator}
Let us recall that the lattice point enumerator measure, $dG_k$, is the measure on $\R^k$ given, for any Borel set $A\subseteq\R^k$, by
$$
G_k(A)=|A\cap\Z^k|,
$$
where $|\cdot|$ denotes cardinality of a set. Whenever $A\subseteq\R^n$ is contained in the affine subspace $x_0+\textrm{span}\{e_1,\dots,e_k\}$ for some $x_0\in\textrm{span}\{e_1,\dots,e_k\}^\perp$, we denote $G_k(A)=G_n(A-x_0)$.

\begin{rmk}\label{rmk:ApproachingVolByDilationsWithGApproachingIntegralWithG}
The measure  $dG_n$ satisfies (see \cite[Lemma 3.22]{TV} and \cite[Section 3.1]{ALY}) that for any convex body $K\subseteq\R^n$ and any bounded set $M$ containing the origin
\begin{equation}\label{eq:ApproachingVolByDilationsWithG}
\lim_{r\to\infty}\frac{G_n(rK+M)}{r^n}=\vol_n(K).
\end{equation}
In particular, taking $M=\{0\}$, for any convex body $K\subseteq\R^n$ we have
\begin{equation}\label{eq:ApproachingVolByDilationsWithGM=0}
\lim_{r\to\infty}\frac{G_n(rK)}{r^n}=\vol_n(K).
\end{equation}
Moreover, for any $f:K\to\R$ which is Riemann-integrable on $K$, we have that
\begin{equation}\label{eq:ApproachingIntegralWithG}
\lim_{r\to\infty}\frac{1}{r^n}\int_{ r K}f\left(\frac{x}{r}\right)dG_n(x)=\lim_{r\to\infty}\frac{1}{r^n}\sum_{x\in K\cap\left(\frac{1}{r}\Z^n\right)}f(x)=\int_{K}f(x)dx,
\end{equation}
where the first identity follows from the definition of the measure $dG_n$. The second identity can be obtained by extending the function $f$ to a rectangle containing $K$ as $f(x)=0$ for every $x\not\in K$, which is Riemann-integrable on the rectangle, and applying \cite[Proposition 6.3]{CS}, which is valid for Riemann-integrable functions on the rectangle. Notice that $\displaystyle{\sum_{x\in K\cap\left(\frac{1}{r}\Z^n\right)}f(x)}$ is a Riemann sum of the extension of $f$ to the rectangle, corresponding to the partition given by the rectangle intersected with cubes with vertices on $\frac{1}{r}\Z^n$.
\end{rmk}

As a consequence of Remark \ref{rmk:ApproachingVolByDilationsWithGApproachingIntegralWithG}, many continuous inequalities can be recovered from discrete inequalities (see, for instance, \cite[Section 3.1]{ALY}). We will see in this paper how Theorem \ref{thm:DiscreteZhangPreIntegration} implies Theorem \ref{thm:ZhangPreIntegration} or how Theorem \ref{thm:PurelyDiscreteZhang} implies Theorem \ref{thm:ZhangPreIntegration2}.

The measure given by the lattice point enumerator satisfies the following discrete version of the Brunn-Minkowski inequality, which was proved in \cite[Theorem 2.1]{IYNZ}, and from which one can recover the classical one \cite[Theorem 7.1.1]{Sch}. It reads as
follows:
\begin{thm}\label{thm: BM_lattice_point_no_G(K)G(L)>0}
Let $\lambda\in(0,1)$ and let $K,L\subset\R^n$ be non-empty bounded sets. Then
$$
G_n\left((1-\lambda)K+\lambda L+C_n\right)^\frac{1}{n}\geq(1-\lambda)G_n(K)^\frac{1}{n}+\lambda G_n(L)^\frac{1}{n}.
$$
\end{thm}
\subsection{Steiner symmetrization}\label{subsec:SteinerSymmetrization}

Given a bounded convex set $K\subseteq\R^n$, the Steiner symmetrization of $K$ with respect to the hyperplane $e_n^\perp$ is defined as
\begin{equation}\label{eq:SteinerSymmetrization}
S_{e_n}(K)=\left\{y+\frac{t_1-t_2}{2}e_n\,:\,y\in P_{e_n^\perp}(K),\,y+t_1e_n\in K,\,y+t_2e_n\in K\right\}.
\end{equation}
That is, $S_{e_n}(K)$ is the set that we obtain by first, shifting all the segments given by $K\cap(x+\langle e_n\rangle)$ in the direction parallel to $\langle e_n\rangle$ until their centers lie in the hyperplane $e_n^\perp$, and second, leaving such segments closed if they were closed and open otherwise. If $K$ is compact then $S_{e_n}(K)$ can be written as
\begin{equation}\label{eq:SteinerSymmetrization2}
S_{e_n}(K)=\left\{(y,t)\in\R^{n-1}\times\R\,:\,y\in P_{e_n^\perp}(K),\,|t|\leq\frac{\vol_1(K\cap(y+\langle e_n\rangle))}{2}\right\}.
\end{equation}
The Steiner symmetrization preserves convexity and volume. Moreover, for every $y\in P_{e_n^\perp}(K)$ we have that $S_{e_n}(K)\cap(y+\langle e_n\rangle)$ is an interval centered at $y$ which has the same length as $K\cap(y+\langle e_n\rangle)$. Besides, from the definition of $S_{e_n}(K)$, if $K\subseteq\R^n$ is a convex body then
$$
S_{e_n}(K)\cap\{x\in\R^n\,:\,\langle x,e_n\rangle\geq0\}
$$
is the hypograph of the function $f:P_{e_n^\perp}(K)\to[0,\infty)$ given by $$f(y)=\frac{\vol_1(K\cap(y+\langle e_n\rangle))}{2},$$ which is concave by Brunn's principle (see, for instance \cite[Theorem 1.2.2]{BGVV}). It is also known that for any convex set $K$ and any $\lambda\geq0$ we have that  $S_{e_n}(\lambda K)=\lambda S_{e_n}(K)$ and that for any two convex sets $K,L\subseteq\R^n$ one has that
\begin{equation}\label{eq:SteinerSymmetrizationInclusion}
S_{e_n}(K)+S_{e_n}(L)\subseteq S_{e_n}(K+L).
\end{equation}
A list of basic properties of the Steiner symmetrization of convex bodies can be found in \cite[Sections 1.1.7 and A.5]{AGM}. Let us point out that for any bounded convex set, since $P_{e_n^\perp}(K)=S_{e_n}(K)\cap e_n^\perp$, then $S_{e_n}(K)\cap\Z^n\neq\emptyset$ if and only if $P_{e_n^\perp}(K)\cap\Z^n\neq\emptyset$, as mentioned in Remark \ref{rmk:Condition0InProjection}.
\subsection{Berwald's inequality}

Berwald's inequality  provides a reverse H\"older's inequality for $L_p$ norms of positive concave functions defined on convex bodies. It is stated in the following theorem:

\begin{thm}[Berwald's inequality]\label{thm:Berwald}
Let $K\subseteq\R^n$ be a convex body and let $f:K\to[0,\infty)$ be a concave function. Then, for any $-1<p\leq q$ we have that
\begin{equation}\label{eq:Berwald'sInequality}
\left(\frac{{{n+q}\choose{n}}}{\vol_n(K)}\int_K f^q(x)dx\right)^\frac{1}{q}\leq\left(\frac{{{n+p}\choose{n}}}{\vol_n(K)}\int_K f^p(x)dx\right)^\frac{1}{p}.
\end{equation}
\end{thm}

Berwald's inequality was proved in \cite[Satz 7]{Be} whenever the parameters in the statement satisfy $0<p<q$ (see also \cite[Theorem 7.2]{AAGJV} for an English translation). It was extended to the whole range $-1<p<q$ in \cite[Theorem 5.1]{GZ}.

For any convex body $K\subseteq\R^n$ and any $\theta\in S^{n-1}$ the function $f_\theta:K\to[0,\infty)$ given by $f_\theta(x)=\vol_1(K\cap\{x+\lambda \theta\,:\lambda\geq0\})$ is concave, as a direct consequence of the convexity of $K$ and the Brunn-Minkowski inequality. Therefore, applying Berwald's inequality \eqref{eq:Berwald'sInequality} for any $-1<p<q$ we obtain the inclusion relation given by \eqref{eq:Inclusionpq} and, as explained in the introduction,
applying inequality \eqref{eq:Berwald'sInequality} to $f_{e_n}$ with parameters $-1<p<q=n$ and taking limits as $p\to(-1)^+$ one obtains Theorem \ref{thm:ZhangPreIntegration}.

With the use of the discrete version of the Brunn-Minkowski inequality, Theorem \ref{thm: BM_lattice_point_no_G(K)G(L)>0}, an analogue of Theorem \ref{thm:Berwald} (in the range $0<p<q$) was proved in \cite[Theorem 1.4]{ALY}, under the condition that the concave function attains its maximum at 0. Before we state it, let us introduce the following notation: if $K\subseteq\R^n$ is a convex body and $f:K\to[0,\infty)$ is a concave function, we denote $f^\diamond:K+C_n\to[0,\infty)$ the function given by
\begin{equation}\label{eq:DefinitionDiamond}
f^\diamond(x)=\sup_{u\in C_n}\overline{f}(x+u),
\end{equation}
where $\overline{f}:\R^n\to[0,\infty)$ is the function given by
$$
\overline{f}(x)=\begin{cases}
f(x) &\textrm{ if }x\in K\\
0 &\textrm{ if }x\not\in K.
\end{cases}
$$
The function $f^\diamond$ satisfies that it is a concave function whose hypograph is the closure of the Minkowski sum of the hypograph of $f$ and $C_n\times\{0\}$.

With this notation, the discrete version of Berwald's inequality reads as follows:
\begin{thm}\label{thm:BerwaldDiscrete}
Let $K\subseteq\R^n$ be a convex body containing the origin and let $f:K\to[0,\infty)$ be a concave function such that $\displaystyle{\max_{x\in K}f(x)=f(0)}$. Then, for any $0<p<q$ we have that
$$
\left(\frac{{{n+q}\choose{n}}}{G_n(K)}\sum_{x\in K\cap\Z^n} f^q(x)dx\right)^\frac{1}{q}\leq\left(\frac{{{n+p}\choose{n}}}{G_n(K)}\sum_{x\in (K+C_n)\cap\Z^n} (f^\diamond)^p(x)dx\right)^\frac{1}{p}.
$$
\end{thm}

Let us point out that the discrete version of Berwald's inequality, Theorem \ref{thm:BerwaldDiscrete}, implies the continuous version of Berwald's inequality, Theorem \ref{thm:Berwald} in the range $0<p<q$. For that matter, see \cite[Theorem 4.5]{ALY} taking into account that in the continuous version of Berwald's inequality, Theorem \ref{thm:Berwald}, we can assume without loss of generality that the concave function attains its maximum at 0.

\subsection{The covariogram function}\label{subsec:Covariogram}

Given a convex body $K\subseteq\R^n$, its covariogram function $g_K:\R^n\to[0,\infty)$ is defined as
$$
g_K(x)=\vol_n(K\cap(x+K)).
$$
The funtion $g_K$ is supported on the difference body of $K$, defined as
\begin{eqnarray}\label{eq:DefinitionDifferenceBody}
K-K&=&\{x-y\,:\,x,y\in K\}=\bigcup_{x\in K}(x-K)\cr
&=&\{x\in\R^n\,:\, K\cap(x+K)\neq\emptyset\}.
\end{eqnarray}
It is clear that $g_K$ is an even function such that $\displaystyle{\max_{x\in\R^n}g_K(x)=g_K(0)=\vol_n(K)}$. Moreover, as a consequence of the Brunn-Minkowski inequality, $g_K^\frac{1}{n}$ is concave on $K-K$ and, by Fubini's theorem,
\begin{eqnarray}\label{eq:IntegralCovariogram}
\int_{\R^n}g_K(x)dx&=&\int_{\R^n}\int_{\R^n}\chi_K(y)\chi_{x+K}(y)dydx=\int_{\R^n}\int_{\R^n}\chi_K(y)\chi_{y-K}(x)dxdy\cr
&=&\int_{\R^n}\chi_K(y)\vol_n(y-K)dy=(\vol_n(K))^2.
\end{eqnarray}
Notice also that, by \eqref{eq:DefinitionDifferenceBody}, for any $\theta\in S^{n-1}$ we have that
\begin{equation}\label{eq:RadialFunctionDifferenceBody}
K\cap(r\theta+K)=\emptyset\quad\textrm{for every }r>\rho_{K-K}(\theta).
\end{equation}
\subsection{Ball bodies of log-concave functions}\label{subsec:BallBodiesLogConcave}

A log-concave function $g:\R^n\to[0,\infty)$ is a function of the form $g(x)=e^{-u(x)}$ with $u:\R^n\to(-\infty,\infty]$ a convex function. The family of log-concave functions plays an extremely important role in the study of problems related to distribution of volume in convex bodies since, as a consequence of the Brunn-Minkowski inequality, the projection of the uniform Lebesgue measure on a convex body in $\R^n$ onto a $k$-dimensional linear subspace is a measure with a log-concave density with respect to the $k$-dimensional Lebesgue measure in that subspace.

Ball introduced in \cite{Ba}, for any measurable (not necessarily log-concave) function $g:\R^n\to[0,\infty)$, such that $g(0)>0$, and for any $p>0$, the set
\begin{equation}\label{eq:Definition p-BallBodies}
K_p(g):=\left\{x\in\R^n\,:\,p\int_0^\infty r^{p-1}g(rx)dr\geq g(0)\right\}.
\end{equation}
Clearly $0\in K_p(g)$, as $\displaystyle{\int_0^\infty r^{p-1}g(0)dr=\infty}$. Besides, for every $x\in \R^n$ and every $\lambda>0$ we have that
$$
p\int_0^\infty r^{p-1}g(r\lambda x)dr=\frac{p}{\lambda^p}\int_0^\infty s^{p-1}g(sx)ds.
$$

Therefore, for every $x\in K_p(g)$ and every $\lambda\in(0,1]$ we have that
$$
p\int_0^\infty r^{p-1}g(r\lambda x)dr=\frac{p}{\lambda^p}\int_0^\infty s^{p-1}g(sx)ds\geq p\int_0^\infty s^{p-1}g(sx)ds\geq g(0)
$$
and $\lambda x\in K_p(g)$. Thus, $K_p(g)$ is a star set with center $0$ whose radial function is given, for any $\theta\in S^{n-1}$, by
$$
\rho_{K_p(g)}(\theta)=\sup\left\{\lambda\geq0\,:\,p\int_0^\infty r^{p-1}g(r\lambda\theta)ds\geq g(0)\right\}.
$$

If $\displaystyle{\int_0^\infty s^{p-1}g(sx)ds=0}$, then $\left\{\lambda\geq0\,:\,p\int_0^\infty r^{p-1}g(r\lambda\theta)ds\geq g(0)\right\}=\{0\}$, and $\rho_{K_p(g)}(\theta)=0$. Otherwise,
\begin{eqnarray*}
\rho_{K_p(g)}(\theta)&=&\sup\left\{\lambda\geq0\,:\,p\int_0^\infty r^{p-1}g(r\lambda\theta)ds\geq g(0)\right\}\cr
&=&\sup\left\{\lambda>0\,:\,\frac{p}{\lambda^p}\int_0^\infty s^{p-1}g(s\theta)ds\geq g(0)\right\}\cr
&=&\left(\frac{p}{g(0)}\int_0^\infty s^{p-1}g(s\theta)ds\right)^\frac{1}{p}.
\end{eqnarray*}
In any case,
\begin{equation}\label{eq:RadialFunctionp-BallBodies}
\rho_{K_p(g)}(\theta)=\left(\frac{p}{g(0)}\int_0^\infty s^{p-1}g(s\theta)ds\right)^\frac{1}{p}.
\end{equation}
The importance of these sets $(K_p(g))_{p>0}$, which we will call $p$-th Ball bodies of $g$, in the study of log-concave functions relies on the following two facts: First, whenever $g:\R^n\to[0,\infty)$ is an integrable log-concave function such that $g(0)>0$, the star set $K_p(g)$ is a convex body for any $p>0$ (see \cite[Theorem 2.5.5, Lemma 2.5.6, and Proposition 2.5.7]{BGVV}). As a particular case, notice that if $K\subseteq\R^n$ is a convex body with $0\in K$, then for any $p>0$
\begin{equation}\label{eq:BallBodiesCharacteristicFunctions}
K_p(\chi_K)=K.
\end{equation}
Second, for any homogeneous function $h:\R^n\to[0,\infty)$ of degree 1 and any $p>-n$ we have, by integration in polar coordinates (see \cite[Proposition 2.5.3]{BGVV} for the particular case when $h$ is a norm on $\R^n$), that
$$
\int_{K_{n+p}(g)}h^p(x)dx=\int_{\R^n}h^p(x)\frac{g(x)}{g(0)}dx.
$$

In particular, if $g:\R^n\to[0,\infty)$ is an integrable log-concave function, such that $g(0)>0$, taking $p=0$ (see \cite[Lemma 2.5.6]{BGVV}) we obtain that  $\displaystyle{\vol_n(K_n(g))=\int_{\R^n}\frac{g(x)}{g(0)}dx}$.

The covariogram function $g_K$ of  a convex body $K\subseteq\R^n$ satisfies that $g_K(0)=\vol_n(K)>0$, $g_K$ is integrable, and $g_K^\frac{1}{n}$ is concave on its support. In particular, $g_K$ is log-concave. Consequently, the radial function defined by \eqref{eq:RadialFunctioNthBallBodyCovariogram} defines the $n$-th ball body of $g_K$, which is a convex body whose volume is, by \eqref{eq:IntegralCovariogram}
\begin{equation}\label{eq:VolumeNthBallBodyCovariogram}
\vol_n(K_n(g_K))=\int_{\R^n}\frac{g_K(x)}{g_K(0)}dx=\vol_n(K).
\end{equation}

Furthermore, for any integrable log-concave function $g:\R^n\to[0,\infty)$ such that $g(0)>0$ the following inclusion relation between Ball bodies holds (see \cite[Proposition 2.5.7]{BGVV}): if $0<p<q$ then
$$
\frac{1}{\Gamma (1+q)^\frac{1}{q}} K_q (g) \subseteq  \frac{1}{\Gamma (1+p)^\frac{1}{p}}K_p (g).
$$

Moreover, since by \cite[Lemma 3.1]{GZ} for any $p>0$ we have  $R_p(K)=K_p(g_K)$, where $R_p(K)$ is defined by \eqref{eq:DefinitionRadialFuncion}, the inclusion relation \eqref{eq:Inclusionpq} shows that if $0<p<q$
\begin{equation}\label{eq:inclusionBallBodiesCovariogram}
{{n+q}\choose{n}}^\frac{1}{q}K_q (g_K)\subseteq{{n+p}\choose{n}}^\frac{1}{p}K_p(g_K).
\end{equation}
This inclusion relation has been extended for the family of $p$-th Ball bodies of $\alpha$-concave functions (i.e., functions such that $f^\alpha$ is concave on its support) with $\alpha>0$ in  \cite[Theorem 1.2]{Ma}.
We refer the reader to \cite[Section 2.5]{BGVV} for more information on the family of $p$-th Ball bodies.

\section{Another proof of Zhang's inequality} \label{sec:NewProof}

In this section we will provide a different proof of Theorem \ref{thm:ZhangPreIntegration}, which leads to Zhang's inequality. In the same way as the proof in \cite{GZ}, it is based on the use of Berwald's inequality (Theorem \ref{thm:Berwald}). However, the choice of the concave function will only require the use of Berwald's inequality with positive parameters. Since Theorem \ref{thm:BerwaldDiscrete} provides a discrete version of Berwald's inequality for positive parameters, we will later be able to use the same approach in order to obtain Theorem \ref{thm:DiscreteZhangPreIntegration}, which provides a discrete version of Theorem \ref{thm:ZhangPreIntegration}.

We begin with the following technical lemma. Part of its proof can be found in \cite[Lemma 2.1 and Lemma 3.1]{GZ}. Nevertheless, we will provide a complete proof for the sake of completeness:

\begin{lemma}\label{lem:IdentitiesIntegral}
Let $K\subseteq\R^n$ be a convex body and let $p>0$. Then
\begin{eqnarray*}
\frac{1}{p+1}\int_{P_{e_n^\perp}(K)}\left(\vol_1(K\cap (y+\langle e_n\rangle))\right)^{p+1}dy&=&p\int_0^{\infty}r^{p-1}\vol_n(K\cap(re_n+K))dr\cr
&=&2^p\int_{S_{e_n}(K)}|\langle x,e_n\rangle|^pdx.
\end{eqnarray*}
\end{lemma}

\begin{proof}
By Fubini's theorem, for any $p>0$ we have
\begin{eqnarray*}
&&p\int_0^{\infty}r^{p-1}\vol_n(K\cap(re_n+K))dr\cr
&=&p\int_0^{\infty}r^{p-1}\int_{P_{e_n^\perp}(K)}\vol_1(K\cap(re_n+K)\cap(y+\langle e_n\rangle))dydr\cr
&=&p\int_0^{\infty}r^{p-1}\int_{P_{e_n^\perp}(K)}\max\{\vol_1(K\cap (y+\langle e_n\rangle))-r,0\}dydr\cr
&=&\int_{P_{e_n^\perp}(K)}\int_0^{\vol_1(K\cap (y+\langle e_n\rangle))}pr^{p-1}(\vol_1(K\cap (y+\langle e_n\rangle))-r)drdy\cr
&=&\int_{P_{e_n^\perp}(K)}\left((\vol_1(K\cap (y+\langle e_n\rangle)))^{p+1}-\frac{p}{p+1}(\vol_1(K\cap (y+\langle e_n\rangle)))^{p+1}\right)dy\cr
&=&\frac{1}{p+1}\int_{P_{e_n^\perp}(K)}\left(\vol_1(K\cap (y+\langle e_n\rangle))\right)^{p+1}dy,
\end{eqnarray*}
which proves the first equality. Besides, we also have that
\begin{eqnarray*}
2^p\int_{S_{e_n}(K)}|\langle x,e_n\rangle|^pdx&=&2^p\int_{P_{e_n^\perp}(K)}\int_{-\frac{\vol_1(K\cap(y+\langle e_n\rangle))}{2}}^{\frac{\vol_1(K\cap(y+\langle e_n\rangle))}{2}}|t|^pdtdy\cr
&=&2^{p+1}\int_{P_{e_n^\perp}(K)}\int_{0}^{\frac{\vol_1(K\cap(y+\langle e_n\rangle))}{2}}t^pdtdy\cr
&=&\frac{1}{p+1}\int_{P_{e_n^\perp}(K)}(\vol_1(K\cap(y+\langle e_n\rangle)))^{p+1}dy,
\end{eqnarray*}
which proves the second equality.
\end{proof}

\begin{rmk}\label{rmk:Equivalence}
Notice that, by the equality between the last two quantities in the statement of Lemma \ref{lem:IdentitiesIntegral} with $p=n$, we obtain that Theorem \ref{thm:DiscreteZhangPreIntegration} can be rewritten as Theorem \ref{thm:ZhangPreIntegration2}.
\end{rmk}

We are now ready to provide our proof of Theorem \ref{thm:ZhangPreIntegration}. 

\begin{proof}[Proof of Theorem \ref{thm:ZhangPreIntegration}]
Let $f: P_{e_n^\perp}(K)\to[0,\infty)$ be the function given by
$$
f(y)=\vol_1(K\cap(y+\langle e_n\rangle)),
$$
which is concave by Brunn's principle \cite[Theorem 1.2.2]{BGVV}. Then, by Berwald's inequality (Theorem \ref{thm:Berwald}) applied on $P_{e_n^\perp}(K)$ with $p=1$ and $q=n+1$ we have that
$$
\left(\frac{{{2n}\choose{n-1}}}{\vol_{n-1}(P_{e_n^\perp}(K))}\int_{P_{e_n^\perp}(K)} f^{n+1}(y)dy\right)^\frac{1}{n+1}\leq\frac{n}{\vol_{n-1}(P_{e_n^\perp}(K))}\int_{P_{e_n^\perp}(K)} f(y)dy.
$$
Equivalently, taking into account that $\frac{1}{n}{{2n}\choose{n-1}}=\frac{1}{n+1}{{2n}\choose{n}}$ and that, by Fubini's theorem, $\displaystyle{\int_{P_{e_n^\perp}(K)} f(y)dy=\vol_n(K)}$, we obtain
$$
\frac{{{2n}\choose{n}}}{n^n}\frac{1}{n+1}\int_{P_{e_n^\perp}(K)} f^{n+1}(y)dy\leq\frac{(\vol_n(K))^{n+1}}{(\vol_{n-1}(P_{e_n^\perp}(K)))^n}.
$$
By Lemma \ref{lem:IdentitiesIntegral} this inequality is equivalent to
$$
\frac{{{2n}\choose{n}}}{n^n}n\int_0^\infty r^{n-1}\vol_n(K\cap(re_n+K))dr\leq\frac{(\vol_n(K))^{n+1}}{(\vol_{n-1}(P_{e_n^\perp}(K)))^n}.
$$
\end{proof}

\begin{rmk}
The main difference between our proof of Theorem \ref{thm:ZhangPreIntegration} and the one in \cite{GZ} is the concave function on which we apply Berwald's inequality (Theorem \ref{thm:Berwald}). In \cite{GZ}, the authors applied Berwald's inequality to the function $f_{e_n}:K\to[0,\infty)$ given by $f_{e_n}(x)=\vol_1(K\cap\{x+\lambda e_n\,:\lambda\geq0\})$ while we considered the function $f: P_{e_n^\perp}(K)\to[0,\infty)$ given by $f(y)=\vol_1(K\cap(y+\langle e_n\rangle))$. With this choice, we do not need to take negative exponents when applying Berwald's inequality  in order to prove the Theorem \ref{thm:ZhangPreIntegration}. The same approach will lead to Theorem \ref{thm:DiscreteZhangPreIntegration} by means of Theorem \ref{thm:BerwaldDiscrete}.
\end{rmk}

An application of Berwald's inequality \eqref{eq:Berwald'sInequality} to the same function $f: P_{e_n^\perp}(K)\to[0,\infty)$ given by $f(y)=\vol_1(K\cap(y+\langle e_n\rangle))$ with exponents $p+1<q+1$ for any $0\leq p<q$ provides the following Theorem, which extends Theorem \ref{thm:ZhangPreIntegration}:
\begin{thm}\label{thm:DifferentInclusion}
Let $K\subseteq\R^n$ be a convex body. For any $0\leq p<q$ we have
$$
\frac{{{n+q}\choose{n}}^\frac{1}{q+1}(n\vol_n(K))^\frac{1}{q+1}}{(\vol_{n-1}(P_{e_n^\perp}(K)))^\frac{1}{q+1}}\rho_{K_q(g_K)}^\frac{q}{q+1}(e_n)\leq
\frac{{{n+p}\choose{n}}^\frac{1}{p+1}(n\vol_n(K))^\frac{1}{p+1}}{(\vol_{n-1}(P_{e_n^\perp}(K)))^\frac{1}{p+1}}\rho_{K_p(g_K)}^\frac{p}{p+1}(e_n).
$$
In particular, taking $p=0$,  for every $q>0$
$$
\frac{{{n+q}\choose{n}}^\frac{1}{q+1}(n\vol_n(K))^\frac{1}{q+1}}{(\vol_{n-1}(P_{e_n^\perp}(K)))^\frac{1}{q+1}}\rho_{K_q(g_K)}^\frac{q}{q+1}(e_n)\leq
\frac{n\vol_n(K)}{\vol_{n-1}(P_{e_n^\perp}(K))}
$$
or, equivalently,
$$
\frac{{{n+q}\choose{n}}}{n^q}q\int_0^\infty r^{q-1}\vol_n(K\cap(re_n+K))dr\leq
\frac{(\vol_n(K))^{q+1}}{(\vol_{n-1}(P_{e_n^\perp}(K)))^q}.
$$
\end{thm}

\begin{rmk}
By considering rotations of a convex body $K\subseteq\R^n$, the last two inequalities in Theorem \ref{thm:DifferentInclusion} being true for every convex body $K$ are also equivalent to the inclusion relation
$$
{{n+q}\choose{n}}^\frac{1}{q}K_q(g_K)\subseteq n\vol_n(K)\Pi^*(K),
$$
for any $q>0$, which is stated in \eqref{eq:Inclusionpq}. More precisely, given a convex body $K\subseteq\R^n$ and $\theta\in S^{n-1}$, applying the last inequality in Theorem \ref{thm:DifferentInclusion} to $U(K)$ with $U\in O(n)$ such that $U^t(e_n)=\theta$ and taking into account, as mentioned in the introduction, that $\vol_n(U(K)\cap(re_n+U(K)))=\vol_n(K\cap(rU^t(e_n)+K))$ and that $P_{e_n^\perp}(U(K))=P_{(U^t(e_n))^\perp}(K)$ we obtain an inequality between $\rho_{K_q(g_K)}(\theta)$ and $\rho_{\Pi^*(K)}(\theta)$ which gives the latter inclusion relation. However, the relation between the convex bodies $K_q(g_K)$ and $K_p(g_K)$ whenever $0<p<q$ given by the first inequality in Theorem \ref{thm:DifferentInclusion} is different from the inclusion relation given by \eqref{eq:Inclusionpq}, as we do not obtain an inequality between radial functions raised to the same power, which would provide an inclusion relation.
\end{rmk}

Finally, we show how Theorem \ref{thm:ZhangPreIntegration} implies Zhang's inequality \eqref{eq:Zhang}

\begin{cor}\label{cor:Zhang}
Let $K\subseteq\R^n$ be a convex body. Then
$$
\frac{{{2n}\choose{n}}}{n^n}\leq(\vol_n(K))^{n-1}\vol_n(\Pi^*K).
$$
\end{cor}

\begin{proof}
By Theorem \ref{thm:ZhangPreIntegration} we have that for any $U\in O(n)$
$$
\frac{{{2n}\choose{n}}}{n^n}n\int_0^\infty r^{n-1}\vol_n(U(K)\cap(re_n+U(K)))dr\leq\frac{(\vol_n(U(K)))^{n+1}}{(\vol_{n-1}(P_{e_n^\perp}(U(K))))^n}.
$$
Equivalently, taking into account that $P_{e_n^\perp}(U(K))=P_{(U^t(e_n))^\perp}(K)$, \\$\vol_n(U(K)\cap(re_n+U(K)))=\vol_n(K\cap(rU^t(e_n)+K))$, and that \\$\vol_n(U(K))=\vol_n(K)$,
$$
\frac{{{2n}\choose{n}}}{n^n}n\int_0^\infty r^{n-1}\vol_n(K\cap(rU^t(e_n)+K))dr\leq\frac{(\vol_n(K))^{n+1}}{(\vol_{n-1}(P_{(U^t(e_n))^\perp}(K)))^n}.
$$
Therefore, since for every $\theta\in S^{n-1}$ there exists $U\in O(n)$ such that $U^t(e_n)=\theta$, we have that for every $\theta\in S^{n-1}$,
$$
\frac{{{2n}\choose{n}}}{n^n}n\int_0^\infty r^{n-1}\vol_n(K\cap(r\theta+K))dr\leq\frac{(\vol_n(K))^{n+1}}{(\vol_{n-1}(P_{\theta^\perp}(K)))^n}.
$$
Taking into account \eqref{eq:RadialFunctioNthBallBodyCovariogram} and that $\rho_{\Pi^*K}(\theta)=\frac{1}{\vol_{n-1}(P_{\theta^\perp}(K))}$ for every $\theta\in S^{n-1}$, we have
$$
{{2n}\choose{n}}^\frac{1}{n}\rho_{K_n(g_K)}(\theta)\leq n\vol_n(K)\rho_{\Pi^*(K)}(\theta)\quad\forall\theta\in S^{n-1},
$$
which is equivalent to the inclusion relation
$$
{{2n}\choose{n}}^\frac{1}{n}K_n(g_K)\subseteq n\vol_n(K)\Pi^*(K)
$$
given by \eqref{inclusion}. Taking volumes and using \eqref{eq:VolumeNthBallBodyCovariogram} we obtain the result.

\end{proof}

\section{A Discrete approach to Zhang's inequality}\label{sec:DiscreteZhang}

In this section we are going to prove Theorem \ref{thm:DiscreteZhangPreIntegration}, which involves the measure $d\mu=dG_{n-1}\otimes dm_1$. As stated in \eqref{eq:DefinitionMu}, for every Borel set $A\in\R^n$
$$
\mu(A)=\sum_{y\in e_n^\perp\cap \Z^n}\vol_1(A\cap(y+\langle e_n\rangle))=\sum_{y\in P_{e_n^\perp}(A)\cap\Z^n}\vol_1(A\cap(y+\langle e_n\rangle)),
$$
where the sum is understood as $0$ if $P_{e_n^\perp}(A)\cap\Z^n=\emptyset$.

Notice that there exist convex bodies $K\subseteq\R^n$ such that $K\cap\Z^n=\emptyset$ and $\mu(K)>0$, as the example $K=[-2,2]\times\left[\frac{1}{3},\frac{1}{2}\right]\subseteq\R^2$ shows. This occurs since, even though $K\cap\Z^n=\emptyset$, we have $P_{e_n^\perp}(K)\cap\Z^n\neq\emptyset$ (and therefore $S_{e_n}(K)\cap\Z^n\neq \emptyset$).

Notice also that there exist convex bodies $K\subseteq\R^n$ such that $K\cap\Z^n\neq\emptyset$ and $\mu(K)=0$, as $K=\left(\frac{1}{2},0\right)+\frac{1}{2}B_2^2\subseteq\R^2$ shows. However, if $K\subseteq\R^n$ is a convex body with $\mu(K)=0$ and $K\cap\Z^n\neq\emptyset$, then necessarily every $x\in K\cap\Z^n$ belongs to $\partial K$. Furthermore, every such $x\in K\cap\Z^n$ satisfies that $K\cap(x+\langle e_n\rangle)=\{x\}$.

Recall that for every bounded convex set $K\subseteq\R^n$, from the definition of the Steiner symmetrization \eqref{eq:SteinerSymmetrization}, we have that $P_{e_n^\perp}(K)\cap\Z^n\neq\emptyset$ if and only if $S_{e_n}(K)\cap\Z^n\neq\emptyset$ and that $P_{e_n^\perp}(K)\cap\Z^n\subseteq S_{e_n}(K)\cap\Z^n$. Besides, $\vol_1(K\cap(y+\langle e_n\rangle))=\vol_1(S_{e_n}(K)\cap(y+\langle e_n\rangle))$ for every $y\in P_{e_n^\perp}(K)$. Therefore, if $P_{e_n^\perp}(K)\cap\Z^n\neq\emptyset$, in the same way as $\vol_n(K)=\vol_n(S_{e_n}(K))$, we have that $\mu(K)=\mu(S_{e_n}(K))$ since
$$
\mu(K)=\sum_{y\in P_{e_n^\perp}(K)\cap\Z^n}\vol_1(K\cap(y+\langle e_n\rangle))=\mu(S_{e_n}(K)).
$$
If $P_{e_n^\perp}(K)\cap\Z^n=\emptyset$, then trivially $\mu(K)=0=\mu(S_{e_n}(K))$. In any case, $\mu(K)=\mu(S_{e_n}(K))$.

Let us also point out that this measure $\mu$ allows to write the right-hand side of the inequality in Theorem \ref{thm:DiscreteZhangPreIntegration} interchanging $K$ by its Steiner symmetrization $S_{e_n}(K)$, providing a smaller upper bound for the integral in the left-hand side, since for every convex body $K\subseteq\R^n$ we have that $\mu(S_{e_n}(K)+C_{n-1})\leq\mu(K+C_{n-1})$, as we show in the following lemma:
\begin{lemma}\label{lem:MuSteiner}
Let $K\subseteq\R^n$ be a bounded convex set. Then, $\mu(S_{e_n}(K))=\mu(K)$ and
$$
\mu(S_{e_n}(K)+C_{n-1})\leq\mu(K+C_{n-1}).
$$
\end{lemma}

\begin{proof}
On the one hand, if $P_{e_n^\perp}(S_{e_n}(K)+C_{n-1})=\emptyset$ then $\mu(S_{e_n}(K)+C_{n-1})=0$ and the inequality is trivial. Let us assume that $P_{e_n^\perp}(S_{e_n}(K)+C_{n-1})\neq\emptyset$.  Notice also that since $C_{n-1}$ is contained in the hyperplane $e_n^\perp$, we have that \\$S_{e_n}(C_{n-1})=C_{n-1}$. Therefore, by \eqref{eq:SteinerSymmetrizationInclusion}, we have that
$$
S_{e_n}(K)+C_{n-1}=S_{e_n}(K)+S_{e_n}(C_{n-1})\subseteq S_{e_n}(K+C_{n-1}).
$$
Since $\mu(S_{e_n}(L))=\mu(L)$ for every bounded convex set $L\subseteq\R^n$ we have that
$$
\mu(S_{e_n}(K)+C_{n-1})\leq\mu(S_{e_n}(K+C_{n-1}))=\mu(K+C_{n-1}).
$$
\end{proof}

Let us recall (see Remark \ref{rmk:ApproachingVolByDilationsWithGApproachingIntegralWithG}) that the measure $dG_n$ satisfies that for any convex body $K\subseteq\R^n$ and any bounded set $M$ containing the origin
$$
\lim_{r\to\infty}\frac{G_n(rK+M)}{r^n}=\vol_n(K).
$$
We continue this section by relating $d\mu$ and $dG_n$ with an error which is controlled by the term $G_{n-1}(P_{e_n^\perp}(K))$ for any convex body $K$. This will imply that the measure $d\mu$ behaves in the same way as the measure $dG_n$ in the limit considered above.

\begin{lemma}\label{lem:muyGn}
Let $K\subseteq\R^n$ be a bounded convex set. Then
$$
G_n(K)-G_{n-1}(P_{e_n^\perp}(K))\leq\mu(K)\leq G_n(K)+G_{n-1}(P_{e_n^\perp}(K)).
$$
Consequently, for any convex body $K\subseteq\R^n$ and any bounded set $M$ containing the origin
$$
\lim_{r\to\infty}\frac{\mu(rK+M)}{r^n}=\vol_n(K).
$$
\end{lemma}

\begin{proof}
First of all, notice that if $P_{e_n^\perp}(K)\cap\Z^n=\emptyset$, then $K\cap\Z^n=\emptyset$. Thus, $G_n(K)=0$, $G_{n-1}(P_{e_n^\perp}(K))=0$ and $\mu(K)=0$ and both inequalities hold. Let us assume that $P_{e_n^\perp}(K)\cap\Z^n\neq\emptyset$. Since $K$ is bounded and convex, for every $y\in P_{e_n^\perp}(K)\cap\Z^n$, if $G_1(K\cap(y+\langle e_n\rangle))=k(y)$, then, if $k(y)\geq 1$, we have $K\cap(y+\langle e_n\rangle)$ contains a segment of length $k(y)-1$ and then
$$
\vol_1(K\cap(y+\langle e_n\rangle))\geq G_1(K\cap(y+\langle e_n\rangle))-1,
$$
while if $k(y)=0$, then the latter inequality is trivial. Furthermore, independently of whether $k(y)=0$ or $k(y)\geq 1$, there exists a segment of length $k(y)+1$ containing $k(y)+2$ points in $\Z^n$, which contains $K\cap(y+\langle e_n\rangle)$. Then
$$
\vol_1(K\cap(y+\langle e_n\rangle))\leq G_1(K\cap(y+\langle e_n\rangle))+1.
$$
\begin{figure}
\begin{center}
\begin{tikzpicture}[>=latex]
\filldraw (0,1) circle (0.05);
\filldraw (1,1) circle (0.05)node[below]{$1$};
\filldraw (2,1) circle (0.05)node[below]{$2$};
\filldraw (3,1) circle (0.05)node[below]{$k(y)$};
\filldraw (4,1) circle (0.05);
\draw(1,1)--(3,1);
\filldraw (0,0) circle (0.05);
\filldraw (1,0) circle (0.05)node[below]{$1$};
\filldraw (2,0) circle (0.05)node[below]{$2$};
\filldraw (3,0) circle (0.05)node[below]{$k(y)$};
\filldraw (4,0) circle (0.05);
\draw (0.3,0)--(3.5,0);
\filldraw (0,-1) circle (0.05);
\filldraw (1,-1) circle (0.05)node[below]{$1$};
\filldraw (2,-1) circle (0.05)node[below]{$2$};
\filldraw (3,-1) circle (0.05)node[below]{$k(y)$};
\filldraw (4,-1) circle (0.05);
\draw(0,-1)--(4,-1);
\end{tikzpicture}
\end{center}
\caption{From top to bottom, we represent the segment of length $k(y)-1$ contained in $K\cap(y+\langle e_n\rangle)$, the segment $K\cap(y+\langle e_n\rangle)$, and the segment of length $k(y)+1$ containing $K\cap(y+\langle e_n\rangle)$.}\label{fig:1}
\end{figure}
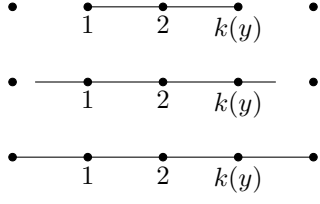
Therefore, summing in $y\in P_{e_n^\perp}(K)\cap\Z^n$, from the definition \eqref{eq:DefinitionMu} of $\mu$ we obtain that
\begin{eqnarray*}
\mu(K)&=&\sum_{y\in P_{e_n^\perp}(K)\cap\Z^n}\vol_1(K\cap(y+\langle e_n\rangle))\cr
&\geq&\sum_{y\in P_{e_n^\perp}(K)\cap\Z^n} G_1(K\cap(y+\langle e_n\rangle))-G_{n-1}(P_{e_n}^\perp(K)\cap\Z^n)
\end{eqnarray*}
and
\begin{eqnarray*}
\mu(K)&=&\sum_{y\in P_{e_n^\perp}(K)\cap\Z^n}\vol_1(K\cap(y+\langle e_n\rangle))\cr
&\leq&\sum_{y\in P_{e_n^\perp}(K)\cap\Z^n} G_1(K\cap(y+\langle e_n\rangle))+G_{n-1}(P_{e_n}^\perp(K)\cap\Z^n).
\end{eqnarray*}
Since $\displaystyle{\sum_{y\in P_{e_n^\perp}(K)\cap\Z^n} G_1(K\cap(y+\langle e_n\rangle))=G_n(K)}$, we obtain that
$$
G_n(K)-G_{n-1}(P_{e_n^\perp}(K))\leq\mu(K)\leq G_n(K)+G_{n-1}(P_{e_n^\perp}(K)).
$$
Finally, taking into account \eqref{eq:ApproachingVolByDilationsWithG} for any bounded set $M$ containing the origin, we obtain that for any convex body $K\subseteq\R^n$
$$
\lim_{r\to\infty}\frac{G_{n-1}(P_{e_n^\perp}(rK+M))}{r^n}=\lim_{r\to\infty}\frac{1}{r}\frac{G_{n-1}(rP_{e_n^\perp}(K)+P_{e_n^\perp}(M))}{r^{n-1}}=0
$$
and then
$$
\lim_{r\to\infty}\frac{\mu(rK+M)}{r^n}=\lim_{r\to\infty}\frac{G_n(rK+M)}{r^n}=\vol_n(K).
$$
\end{proof}

The following lemma is the analogue of Lemma \ref{lem:IdentitiesIntegral} when the measure $\mu$, instead of the Lebesgue measure, is considered. It is proved in the same way, since both $dm_n=dm_{n-1}\otimes dm_1$ and $d\mu=dG_{n-1}\otimes dm_1$ are product measures on $\R^n=\R^{n-1}\times\R$ where the second factor in both measures is the $1$-dimensional Lebesgue measure. We include the proof here for the sake of completeness.

\begin{lemma}\label{lem:IdentitiesSum}
Let $K\subseteq\R^n$ be a convex body and $p>0$. Then
\begin{eqnarray*}
\frac{1}{p+1}\sum_{y\in P_{e_n^\perp}(K)\cap\Z^n}\left(\vol_1(K\cap (y+\langle e_n\rangle))\right)^{p+1}&=&p\int_0^{\infty}r^{p-1}\mu(K\cap(re_n+K))dr\cr
&=&2^p\int_{S_{e_n}(K)}|\langle x,e_n\rangle|^pd\mu(x).
\end{eqnarray*}
\end{lemma}

\begin{proof}
First of all, notice that if $P_{e_n^\perp}(K)\cap\Z^n=\emptyset$, then $\mu(K)=0$, and therefore, for every $r\geq 0$, we have that $\mu(K\cap(re_n+K))=0$. In such case, we also have  $P_{e_n^\perp}(S_{e_n}(K))\cap\Z^n=\emptyset$ and $\mu(S_{e_n}(K))=0$. Thus, all the identities are trivial. Let us assume that $P_{e_n^\perp}(K)\cap\Z^n\neq\emptyset$. From the definition of $\mu$ given in \eqref{eq:DefinitionMu}, we have that for any $r\geq0$
\begin{eqnarray*}
\mu(K\cap(re_n+K))&=&\sum_{y\in P_{e_n^\perp}(K)\cap\Z^n}\vol_1(K\cap(re_n+K)\cap(y+\langle e_n\rangle))\cr
&=&\sum_{y\in P_{e_n^\perp}(K)\cap\Z^n}\max\{\vol_1(K\cap (y+\langle e_n\rangle))-r,0\}.
\end{eqnarray*}
Then, for any $p>0$ we have
\begin{eqnarray*}
&&p\int_0^{\infty}r^{p-1}\mu(K\cap(re_n+K))dr\cr
&=&p\int_0^{\infty}r^{p-1}\sum_{y\in P_{e_n^\perp}(K)\cap\Z^n}\max\{\vol_1(K\cap (y+\langle e_n\rangle))-r,0\}dr\cr
&=&\sum_{y\in P_{e_n^\perp}(K)\cap\Z^n}\int_0^{\vol_1(K\cap (y+\langle e_n\rangle))}pr^{p-1}(\vol_1(K\cap (y+\langle e_n\rangle))-r)dr\cr
&=&\sum_{y\in P_{e_n^\perp}(K)\cap\Z^n}\left((\vol_1(K\cap (y+\langle e_n\rangle)))^{p+1}-\frac{p}{p+1}(\vol_1(K\cap (y+\langle e_n\rangle)))^{p+1}\right)\cr
&=&\frac{1}{p+1}\sum_{y\in P_{e_n^\perp}(K)\cap\Z^n}\left(\vol_1(K\cap (y+\langle e_n\rangle))\right)^{p+1},
\end{eqnarray*}
which proves the first identity. Finally, notice that
\begin{eqnarray*}
2^p\int_{S_{e_n}(K)}|\langle x,e_n\rangle|^pd\mu(x)&=&2^p\sum_{y\in P_{e_n^\perp}(K)\cap\Z^n}\int_{-\frac{\vol_1(K\cap(y+\langle e_n\rangle))}{2}}^{\frac{\vol_1(K\cap(y+\langle e_n\rangle))}{2}}|t|^pdt\cr
&=&2^{p+1}\sum_{y\in P_{e_n^\perp}(K)\cap\Z^n}\int_0^{\frac{\vol_1(K\cap(y+\langle e_n\rangle))}{2}}t^pdt\cr
&=&\frac{1}{p+1}\sum_{y\in P_{e_n^\perp}(K)\cap\Z^n}(\vol_1(K\cap(y+\langle e_n\rangle)))^{p+1},\cr
\end{eqnarray*}
which proves the second identity.
\end{proof}

Before proving Theorem \ref{thm:DiscreteZhangPreIntegration}, let us recall that if $K\subseteq\R^n$ is a convex body and $f:K\to[0,\infty)$ is a concave function, the function $f^\diamond:K+C_n\to[0,\infty)$, defined  in \eqref{eq:DefinitionDiamond}, is the concave function whose hypograph is the closure of the Minkowski sum of the hypograph of $f$ and $C_n\times\{0\}$.

\begin{proof}[Proof of Theorem \ref{thm:DiscreteZhangPreIntegration}]
Let $K\subseteq\R^n$ be a convex body satisfying $$\max_{y\in e_n^\perp}\vol_1(K\cap(y+\langle e_n\rangle))=\vol_1(K\cap\langle e_n\rangle),$$ and let
\begin{eqnarray*}
K_1&=&S_{e_n}(K)\cap\{x\in\R^n\,:\,\langle x,e_n\rangle\geq0\}\cr
&=&\left\{(y,t)\in\R^{n-1}\times\R\,:\,y\in P_{e_n^\perp}(K),\,0\leq t\leq\frac{\vol_1(K\cap(y+\langle e_n\rangle))}{2}\right\}.
\end{eqnarray*}
Notice that $K_1$ is the hypograph of the concave function $f:P_{e_n^\perp}(K)\to[0,\infty)$ given by
$$
f(y)=\frac{\vol_1(K\cap(y+\langle e_n\rangle))}{2}.
$$
By Theorem \ref{thm:BerwaldDiscrete} applied on $P_{e_n^\perp}(K)$ with $p=1$ and $q=n$, we have
\begin{eqnarray*}
&&\left(\frac{{{2n}\choose{n-1}}}{G_{n-1}(P_{e_n^\perp}(K))}\sum_{y\in P_{e_n^\perp}(K)\cap\Z^n}f^{n+1}(y)\right)^\frac{1}{n+1}\cr
&\leq&\frac{n}{G_{n-1}(P_{e_n^\perp}(K))}\sum_{y\in (P_{e_n^\perp}(K)+C_{n-1})\cap\Z^n}f^{\diamond}(y).
\end{eqnarray*}
Equivalently, taking into account that $\frac{1}{n}{{2n}\choose{n-1}}=\frac{1}{n+1}{{2n}\choose{n}}$, we obtain that
\begin{eqnarray*}
&&\frac{{{2n}\choose{n}}}{n^n}\sum_{y\in P_{e_n^\perp}(K)\cap\Z^n}\frac{\vol_1(K\cap(y+\langle e_n\rangle))^{n+1}}{n+1}\cr
&\leq&\frac{1}{G_{n-1}(P_{e_n^\perp}(K))^n}\left(\sum_{y\in (P_{e_n^\perp}(K)+C_{n-1})\cap\Z^n}2f^{\diamond}(y)\right)^{n+1}.
\end{eqnarray*}
On the one hand, by Lemma \ref{lem:IdentitiesSum},
$$
\frac{1}{n+1}\sum_{y\in P_{e_n^\perp}(K)}\vol_1(K\cap(y+\langle e_n\rangle))^{n+1}=n\int_0^{\infty}r^{n-1}\mu(K\cap(re_n+K))dr.
$$
On the other hand, notice that the hypograph of $f^\diamond$ is the closure of
\begin{eqnarray*}
L_1&:=&K_1+C_{n-1}=S_{e_n}(K)\cap\{x\in\R^n\,:\,\langle x,e_n\rangle\geq0\}+C_{n-1}\cr
&=&(S_{e_n}(K)+C_{n-1})\cap\{x\in\R^n\,:\,\langle x,e_n\rangle\geq0\},
\end{eqnarray*}
and then
$$
\sum_{y\in (P_{e_n^\perp}(K)+C_{n-1})\cap\Z^n}2f^{\diamond}(y)=2\mu(L_1)=\mu(S_{e_n}(K)+C_{n-1}).
$$
Therefore,
$$
\frac{{{2n}\choose{n}}}{n^n}n\int_0^\infty r^{n-1}\mu(K\cap(re_n+K))dr\leq\frac{(\mu(S_{e_n}(K)+C_{n-1}))^{n+1}}{(G_{n-1}(P_{e_n^\perp}(K)))^n}.
$$
\end{proof}

\begin{rmk}
Let us recall that in Theorem \ref{thm:ZhangPreIntegration} the direction $e_n$ does not play any special role and, as mentioned in the introduction, having Theorem \ref{thm:ZhangPreIntegration} for any convex body is equivalent to having inequality \eqref{inclusion2} for any convex body and any direction $\theta\in S^{n-1}$. However, in the discrete case the coordinate direction $e_n$ plays a special role since, at the core of the proof of the discrete version of Berwald's inequality (Theorem \ref{thm:BerwaldDiscrete}), it is essential that $P_{e_n^\perp}\Z^n$ can be identified with $\Z^{n-1}$ in order to apply the discrete version of Brunn-Minkowski inequality (Theorem \ref{thm: BM_lattice_point_no_G(K)G(L)>0}) in $e_n^\perp$ identified with $\R^{n-1}$. However, this is not the case for $P_{\theta^\perp}\Z^n$ with a generic $\theta\in S^{n-1}$. The special role of the direction $e_n$ is also reflected in the definition of the measure $\mu$ as $d\mu=dG_{n-1}\otimes dm_1$, which depends on the choice of the direction $e_n$. If we chose a different direction $\theta$, a definition of the measure $\mu$ analogue to the one given by \eqref{eq:DefinitionMu} (or even with the sum in $y\in P_{\theta^\perp}(K\cap\Z^n)$) would not be identified with a product measure where the first factor is $dG_{n-1}$, since $P_{\theta^\perp}\Z^n$ would not be identified with $\Z^{n-1}$. This fact would not allow the use of the discrete version of Brunn-Minkowski inequality (Theorem \ref{thm: BM_lattice_point_no_G(K)G(L)>0}) in $e_n^\perp$ identified with $\R^{n-1}$.
\end{rmk}

\begin{rmk}
Let us also point out that, even though the direction $e_n$ plays a special role in the definition of the measure $\mu$, by Lemma \ref{lem:muyGn} the measure $d\mu$ is closely related to the discrete measure $dG_n$, which is the counting measure on $\Z^n$ and, taking into account Remark \ref{rmk:ApproachingVolByDilationsWithGApproachingIntegralWithG}, we will be able to recover Theorem \ref{thm:ZhangPreIntegration} from Theorem \ref{thm:DiscreteZhangPreIntegration}.
\end{rmk}
Let us now obtain Corollary \ref{cor:ZhangPreIntegrationLatticePoint}, where the only measures involved are the ones given by the lattice point enumerator:

\begin{proof}[Proof of Corollary \ref{cor:ZhangPreIntegrationLatticePoint}]
Let $K\subseteq\R^n$ be a convex body satisfying \\$\displaystyle{\max_{y\in e_n^\perp}\vol_1(K\cap(y+\langle e_n\rangle))}=\vol_1(K\cap\langle e_n\rangle)$. By Theorem \ref{thm:DiscreteZhangPreIntegration} we have that
$$
\frac{{{2n}\choose{n}}}{n^n}n\int_0^\infty r^{n-1}\mu(K\cap(re_n+K))dr\leq\frac{(\mu(S_{e_n}(K)+C_{n-1}))^{n+1}}{(G_{n-1}(P_{e_n^\perp}(K)))^n}.
$$
On the one hand, by Lemma \ref{lem:muyGn} we have that
\begin{eqnarray*}
\mu(S_{e_n}(K)+C_{n-1})&\leq& G_n(S_{e_n}(K)+C_{n-1})+G_{n-1}(P_{e_n^\perp}(S_{e_n}(K)+C_{n-1}))\cr
&=&G_n(S_{e_n}(K)+C_{n-1})+G_{n-1}(P_{e_n^\perp}(K)+C_{n-1}).
\end{eqnarray*}
On the other hand, by \eqref{eq:RadialFunctionDifferenceBody}, we have that $K\cap(re_n+K)=\emptyset$ if $r>\rho_{K-K}(e_n)$. Therefore, using Lemma \ref{lem:muyGn} and the fact that $K\cap(re_n+K)\subseteq K$ we have
\begin{eqnarray*}
&&n\int_0^\infty r^{n-1}\mu(K\cap(re_n+K))dr=n\int_0^{\rho_{K-K}(e_n)} r^{n-1}\mu(K\cap(re_n+K))dr\cr
&\geq&n\int_0^{\rho_{K-K}(e_n)} r^{n-1}(G_n(K\cap(re_n+K))-G_{n-1}(P_{e_n^\perp}(K\cap(re_n+K))))dr\cr
&\geq&n\int_0^{\rho_{K-K}(e_n)} r^{n-1}(G_n(K\cap(re_n+K))-G_{n-1}(P_{e_n^\perp}(K)))dr\cr
&=&n\int_0^{\rho_{K-K}(e_n)} r^{n-1}G_n(K\cap(re_n+K))dr-\rho_{K-K}^n(e_n)G_{n-1}(P_{e_n^\perp}(K))\cr
&=&n\int_0^{\infty} r^{n-1}G_n(K\cap(re_n+K))dr-\rho_{K-K}^n(e_n)G_{n-1}(P_{e_n^\perp}(K)).\cr
\end{eqnarray*}
\end{proof}

Finally, we are going to see that Theorem \ref{thm:DiscreteZhangPreIntegration} implies Theorem \ref{thm:ZhangPreIntegration} and therefore, as seen in the proof of Corollary \ref{cor:Zhang}, Zhang's inequality \eqref{eq:Zhang}.

\begin{cor}[Theorem \ref{thm:ZhangPreIntegration}]\label{cor:DiscreteToContinuous}
Let $K\subseteq\R^n$ be a convex body. Then,
$$
\frac{{{2n}\choose{n}}}{n^n}\int_0^\infty nr^{n-1}\vol_n(K\cap(re_n+ K))dr\leq\frac{\vol_n(K)^{n+1}}{\vol_{n-1}(P_{e_n^\perp}(K))^n}.
$$
\end{cor}

\begin{proof}
Let $K\subseteq\R^n$ be a convex body. Since the inequality we want to prove is invariant by translations, we can assume without loss of generality that \\$\displaystyle{\max_{y\in P_{e_n^\perp}(K)}=\vol_1(K\cap(y+\langle e_n\rangle))=\vol_1(K\cap\langle e_n\rangle)}$. Then, for any $\lambda>0$ we have that
$$
\max_{y\in P_{e_n^\perp}(\lambda K)}\vol_1(\lambda K\cap(y+\langle e_n\rangle))=\vol_1(\lambda K\cap\langle e_n\rangle).
$$
Thus, by Theorem \ref{thm:DiscreteZhangPreIntegration} and Lemma \ref{lem:MuSteiner}, for any $\lambda>0$,
\begin{eqnarray*}
\frac{{{2n}\choose{n}}}{n^n}\int_0^\infty nr^{n-1}\mu(\lambda K\cap(re_n+\lambda K))dr&\leq&\frac{(\mu(S_{e_n}(\lambda K)+C_{n-1}))^{n+1}}{(G_{n-1}(P_{e_n^\perp}(\lambda K)))^n}\cr
&\leq&\frac{\mu(\lambda K+C_{n-1})^{n+1}}{G_{n-1}(P_{e_n^\perp}(\lambda K))^n}.
\end{eqnarray*}
By Lemma \ref{lem:muyGn} we have that for every $r>0$
$$
\mu(\lambda K\cap(re_n+\lambda K))\geq G_n(\lambda K\cap(re_n+\lambda K))-G_{n-1}(P_{e_n^\perp}(\lambda K\cap(re_n+\lambda K)))
$$
and that
$$
\mu(\lambda K+C_{n-1})\leq G_n(\lambda K+C_{n-1})+G_{n-1}(\lambda P_{e_n^\perp}(K)+C_{n-1}).
$$
Then, we obtain that
\begin{eqnarray*}
&&\frac{{{2n}\choose{n}}}{n^n}\int_0^\infty nr^{n-1}(G_n(\lambda K\cap(re_n+\lambda K))-G_{n-1}(P_{e_n^\perp}(\lambda K\cap(re_n+\lambda K))))dr\cr
&\leq&\frac{(G_n(\lambda K+C_{n-1})+G_{n-1}(\lambda P_{e_n^\perp}(K)+C_{n-1}))^{n+1}}{G_{n-1}(\lambda P_{e_n^\perp}(K))^n}.
\end{eqnarray*}
Therefore, dividing the latter inequality by $\lambda^{2n}=\frac{\lambda^{n(n+1)}}{\lambda^{n(n-1)}}$, we obtain that for any $\lambda>0$, the quantity
\begin{equation}\label{eq:InequaliProposition41Left}
\frac{{{2n}\choose{n}}}{n^n}\frac{1}{\lambda^{2n}}\int_0^\infty nr^{n-1}(G_n(\lambda K\cap(re_n+\lambda K))-G_{n-1}(P_{e_n^\perp}(\lambda K\cap(re_n+\lambda K))))dr
\end{equation}
is bounded above by
$$
\frac{(G_n(\lambda K+C_{n-1})+G_{n-1}(\lambda P_{e_n^\perp}(K)+C_{n-1}))^{n+1}}{\lambda^{2n}G_{n-1}(\lambda P_{e_n^\perp}(K))^n},
$$
which equals
\begin{equation}\label{eq:InequaliProposition41}
\frac{\left(\frac{G_n(\lambda K+C_{n-1})}{\lambda^n}+\frac{G_{n-1}(\lambda P_{e_n^\perp}(K)+C_{n-1})}{\lambda^n}\right)^{n+1}}{\left(\frac{G_{n-1}(\lambda P_{e_n^\perp}(K))}{\lambda^{n-1}}\right)^n}.
\end{equation}

On the one hand, we are going to prove that
\begin{eqnarray}\label{eq:Claim1}
&&\lim_{\lambda\to\infty}\frac{1}{\lambda^{2n}}\int_0^\infty nr^{n-1}G_n(\lambda K\cap(re_n+\lambda K))dr\cr
&=&\frac{{{2n}\choose{n}}}{n^n}\int_0^\infty nr^{n-1}\vol_n( K\cap(re_n+ K))dr
\end{eqnarray}
and that
\begin{equation}\label{eq:Claim2}
\lim_{\lambda\to\infty}\frac{1}{\lambda^{2n}}\int_0^\infty nr^{n-1}G_{n-1}(P_{e_n^\perp}(\lambda K\cap(re_n+\lambda K)))dr=0.
\end{equation}
As a consequence, we will obtain that \eqref{eq:InequaliProposition41Left} converges, as $\lambda\to\infty$, to
$$
\frac{{{2n}\choose{n}}}{n^n}\int_0^\infty nr^{n-1}\vol_n( K\cap(re_n+ K))dr.
$$
On the other hand, we have that, by \eqref{eq:ApproachingVolByDilationsWithGM=0},
$$
\lim_{\lambda\to\infty}\frac{G_{n-1}(\lambda P_{e_n^\perp}(K))}{\lambda^{n-1}}=\vol_{n-1}(P_{e_n^\perp}(K))
$$
and, by \eqref{eq:ApproachingVolByDilationsWithG} (with $M=C_{n-1}$)
\begin{eqnarray*}
&&\lim_{\lambda\to\infty}\left(\frac{G_n(\lambda K+C_{n-1})}{\lambda^n}+\frac{G_{n-1}(\lambda P_{e_n^\perp}(K)+C_{n-1})}{\lambda^n}\right)\cr
&=&\vol_n(K)+0\cdot\vol_{n-1}(P_{e_n^\perp}(K))=\vol_n(K).
\end{eqnarray*}
Therefore, \eqref{eq:InequaliProposition41} converges to $\frac{\vol_n(K)^{n+1}}{\vol_{n-1}(P_{e_n^\perp}(K))^n}$ as $\lambda\to\infty$. As a consequence, we will obtain
$$
\frac{{{2n}\choose{n}}}{n^n}\int_0^\infty nr^{n-1}\vol_n( K\cap(re_n+ K))dr\leq\frac{\vol_n(K)^{n+1}}{\vol_{n-1}(P_{e_n^\perp}(K))^n}.
$$

Let us prove \eqref{eq:Claim1}: By \eqref{eq:RadialFunctionDifferenceBody}, $\lambda K\cap(re_n+\lambda K)=\emptyset$ for every $r>\rho_{\lambda(K-K)}(e_n)$. Therefore, changing variables $r=\lambda s$, for any $\lambda>0$ we have
\begin{eqnarray*}
&&\frac{1}{\lambda^{2n}}\int_0^\infty nr^{n-1}G_n(\lambda K\cap(re_n+\lambda K))dr=\cr
&=&\frac{1}{\lambda^{2n}}\int_0^{\rho_{\lambda(K-K)}(e_n)} nr^{n-1}G_n\left(\lambda\left( K\cap\left(\frac{r}{\lambda} e_n+ K\right)\right)\right)dr\cr
&=&\int_0^{\rho_{K-K}(e_n)} ns^{n-1}\frac{G_n(\lambda(K\cap(s e_n+ K)))}{\lambda^n}ds.\cr
\end{eqnarray*}
For any $s\in[0,\rho_{K-K}(e_n)]$ we have by \eqref{eq:ApproachingVolByDilationsWithGM=0} that
$$
\lim_{\lambda\to\infty}\frac{G_n(\lambda(K\cap(s e_n+ K)))}{\lambda^n}=\vol_n(K\cap (se_n+K)).
$$
Also by \eqref{eq:ApproachingVolByDilationsWithGM=0}
$$
\lim_{\lambda\to\infty}\frac{G_n(\lambda K)}{\lambda^n}=\vol_n(K).
$$
Thus, given $\varepsilon_0>0$ there exists $\lambda_0>0$ such that if $\lambda>\lambda_0$
$$
ns^{n-1}\frac{G_n(\lambda(K\cap(s e_n+ K)))}{\lambda^n}\leq ns^{n-1}\frac{G_n(\lambda K)}{\lambda^n}\leq ns^{n-1}(\vol_n(K)+\varepsilon_0).
$$
The function $f(s)=ns^{n-1}(\vol_n(K)+\varepsilon_0)$ is integrable in $[0,\rho_{K-K}(e_n)]$. Thus, by the dominated convergence theorem and using again that, by \eqref{eq:RadialFunctionDifferenceBody}, $K\cap(se_n+ K)=\emptyset$ if $s>\rho_{K-K}(e_n)$, we obtain
\begin{eqnarray*}
&&\lim_{\lambda\to\infty}\frac{1}{\lambda^{2n}}\int_0^\infty nr^{n-1}G_n(\lambda K\cap(re_n+\lambda K))dr\cr
&=&\int_0^{\rho_{K-K}(e_n)}ns^{n-1}\vol_n(K\cap(se_n+ K))ds\cr
&=&\int_0^{\infty}ns^{n-1}\vol_n(K\cap(se_n+ K))ds,\cr
\end{eqnarray*}
which, renaming $s$ as $r$, proves \eqref{eq:Claim1}.

Let us now prove \eqref{eq:Claim2}: Again, by \eqref{eq:RadialFunctionDifferenceBody}, we have that $\lambda K\cap(re_n+\lambda K)=\emptyset$ if $r>\rho_{\lambda(K-K)}(e_n)$ and then, for any $\lambda>0$ we have
\begin{eqnarray*}
0&\leq&\frac{1}{\lambda^{2n}}\int_0^\infty nr^{n-1}G_{n-1}(P_{e_n^\perp}(\lambda K\cap(re_n+\lambda K)))dr\cr
&=&\frac{1}{\lambda^{2n}}\int_0^{\rho_{\lambda(K-K)}(e_n)} nr^{n-1}G_{n-1}(P_{e_n^\perp}(\lambda K\cap(re_n+\lambda K)))dr\cr
&\leq& \frac{G_{n-1}(\lambda P_{e_n^\perp}(K))}{\lambda^{2n}}\int_0^{\rho_{\lambda(K-K)}(e_n)}nr^{n-1}dr=\frac{G_{n-1}(\lambda P_{e_n^\perp}(K))\rho_{\lambda(K-K)}^n(e_n)}{\lambda^{2n}}\cr
&=&\frac{1}{\lambda}\frac{G_{n-1}(\lambda P_{e_n^\perp}(K))}{\lambda^{n-1}}\rho_{K-K}^n(e_n).
\end{eqnarray*}
Since by \eqref{eq:ApproachingVolByDilationsWithGM=0}
$$
\lim_{\lambda\to\infty}\frac{1}{\lambda}\frac{G_{n-1}(\lambda P_{e_n^\perp}(K))}{\lambda^{n-1}}\rho_{K-K}^n(e_n)=0\cdot\vol_{n-1}(P_{e_n^\perp}(K))\rho_{K-K}^n(e_n)=0,
$$
we have that
$$
\lim_{\lambda\to\infty}\frac{1}{\lambda^{2n}}\int_0^\infty nr^{n-1}G_{n-1}(P_{e_n^\perp}(\lambda K\cap(re_n+\lambda K)))dr=0.
$$
Now that we have proved \eqref{eq:Claim1} and \eqref{eq:Claim2}, the proof is complete.
\end{proof}

\section{A different discrete approach to Zhang's inequality}\label{sec:AnotherDiscreteApproach}

In this section we are going to prove Theorem \ref{thm:PurelyDiscreteZhang}. This gives a discrete version of Theorem \ref{thm:ZhangPreIntegration2}. Let us recall that, as mentioned in Remark \ref{rmk:Equivalence}, Theorem \ref{thm:ZhangPreIntegration2} is equivalent to Theorem \ref{thm:ZhangPreIntegration}. The proof will follow the lines of the proof of Theorem \ref{thm:BerwaldDiscrete}, which was given in \cite{ALY}. For any convex body $K\subseteq\R^n$ such that $P_{e_n^\perp}(K)\cap\Z^n\neq\emptyset$, the role of the function $f$ in Theorem \ref{thm:BerwaldDiscrete} will be played by the function $f_1:P_{e_n^\perp}(K)\to[0,\infty)$ given by $f_1(y)=\frac{1}{2}\vol_1(K\cap(y+\langle e_n\rangle))$, whose hypograph will be $S_{e_n}(K)\cap\{x\in\R^n\,:\,\langle x,e_n\rangle\geq0\}$. Therefore, the hypograph of the function $f_1^\diamond$, defined as in \eqref{eq:DefinitionDiamond}, is the closure of $(S_{e_n}(K)+C_{n-1})\cap\{x\in\R^n\,:\,\langle x,e_n\rangle\geq0\}$. The role of the binomial coefficient in Theorem \ref{thm:BerwaldDiscrete} will now be played by the inverse of the number, defined in \eqref{eq:DefinitionBm(p)} as
$$
B_{m}(p)=\sum_{k=0}^{\lfloor m\rfloor}\frac{p}{m}\left(1-\frac{k}{m}\right)^{n-1}\left(\frac{k}{m}\right)^{p-1},
$$
for a certain $m=m_0(p)>1$, whose existence (depending on a fixed parameter $p\geq1$) needs to be proved. Once the existence of such $m_0(p)$ is proved, we will be able to construct an appropriate $\frac{1}{n-1}$-affine function (i.e., a function $g_p$ such that $g_p^\frac{1}{n-1}$ is affine on its support), given by
\begin{equation}\label{eq:Definitiong_p}
g_p(x)=\begin{cases}
\left(1-\frac{x}{m_0(p)}\right)^{n-1}G_{n-1}(P_{e_n^\perp}(K)) &\textrm{ if }0\leq x\leq m_0(p)\cr
0&\textrm{ if }x>m_0(p).\cr
\end{cases}
\end{equation}
This function $g_p$ which will have a crossing point $r_0(p)$ (see Lemma \ref{lem:existenceDiscreter0} below for the precise definition of such crossing point) with the function  $\tilde{f}$ given by
\begin{equation}\label{eq:DefinitionfTilde}
\tilde{f}(r)=G_{n-1}((S_{e_n}(K)+C_{n-1})\cap\{x\in\R^n\,:\,\langle x,e_n\rangle=r\}),\quad r\geq0.
\end{equation}
which is a modification of the function given by
\begin{equation}\label{eq:DefinitionfSinTilde}
f(r)=G_{n-1}(S_{e_n}(K)\cap\{x\in\R^n\,:\,\langle x,e_n\rangle=r\}),\quad r\geq0.
\end{equation}
The function $f$ counts, for every $r\geq0$, the number of integer points in the projection onto $\{x\in\R^n\,:\,\langle x,e_n\rangle=0\}$ of the intersection of $S_{e_n}(K)$ with the hyperplane $\{x\in\R^n\,:\,\langle x,e_n\rangle=r\}$.

Besides, such function $g_p$ will satisfy that, the value of $\displaystyle{\left(\frac{B_{m}(q)^{-1}}{G_{n-1}(P_{e_n^\perp}(K))}\sum_{k=0}^\infty qk^{q-1}g_p(k)\right)^\frac{1}{q}}$ will be independent of $q$ and will depend only on the fixed parameter $p$. For every $q\geq1$  we will have (See equations \eqref{eq:Sumg_p} and \eqref{eq:m_0(p)} below and Lemma \ref{lem:existenceDiscreter0} for the precise definition of the function $g_p$)
$$
\left(\frac{B_{m}(q)^{-1}}{G_{n-1}(P_{e_n^\perp}(K))}\sum_{k=0}^\infty qk^{q-1}g_p(k)\right)^\frac{1}{q}=\left(\frac{B_{m}(p)^{-1}}{G_{n-1}(P_{e_n^\perp}(K))}\sum_{k=0}^\infty pk^{q-1}\tilde{f}(k)\right)^\frac{1}{q},
$$
Making use of the fact that $g_p$ will have a crossing point with $\tilde{f}$, we will be able to prove that
$$
\left(\frac{B_{m}(q)^{-1}}{G_{n-1}(P_{e_n^\perp}(K))}\sum_{k=0}^\infty qk^{q-1}g_p(k)\right)^\frac{1}{q}\geq\left(\frac{B_{m}(q)^{-1}}{G_{n-1}(P_{e_n^\perp}(K))}\sum_{k=0}^\infty qk^{q-1}f(k)\right)^\frac{1}{q},
$$
obtaining in this way the proof of Theorem \ref{thm:PurelyDiscreteZhang}.

Let us point out that, from the definition of $B_{m}(p)$, we have that if $p>1$, then $B_m(p)=0$ for every $m\in(0,1)$ and if $p=1$, then $B_m(1)=\frac{1}{m}$ for every $m\in(0,1)$, since we convene that $\left(\frac{0}{m}\right)^{1-1}=1$. For any $p\geq 1$ and $m>1$ we have that $B_m(p)>0$.

Notice that Berwald's inequality, \eqref{eq:Berwald'sInequality}, applied in $e_n^\perp$ identified  with $\R^{n-1}$ to a concave function $h$ defined on $P_{e_n^\perp}K$, shows that the quantity $$\left(\frac{{{n-1+p}\choose{n-1}}}{\vol_{n-1}(P_{e_n^\perp}(K))}\int_{P_{e_n^\perp}(K)}h^p(x)dx\right)^\frac{1}{p}$$ is non-increasing in $p\in(-1,\infty)$.

The following lemma shows that, as $m$ tends to $\infty$, the value of $B_{m}(p)$ converges to the inverse of ${{n-1+p}\choose{n-1}}$, which is the binomial coefficient in this quantity.
\begin{lemma}\label{lem:limit}
For any $n\geq2$ and any $p\geq1$ we have that
$$
\lim_{x\to\infty} B_{x}(p)=\lim_{x\to\infty}\sum_{k=0}^{\lfloor x\rfloor}\frac{p}{x}\left(1-\frac{k}{x}\right)^{n-1}\left(\frac{k}{x}\right)^{p-1}={{n-1+p}\choose{n-1}}^{-1}.
$$
\end{lemma}

\begin{proof}
Let $f:[0,1]\to\R$ be the function $f(x)=p(1-x)^{n-1}x^{p-1}$, which is Riemann-integrable in $[0,1]$ with
$$
\int_0^1 f(x)dx={{n-1+p}\choose{n-1}}^{-1}.
$$
Since the norm of the partition of the interval $[0,1]$ given by $\mathcal{P}_x=\left\{0,\frac{1}{\lfloor x\rfloor},\frac{2}{\lfloor x\rfloor},\dots,1\right\}$ tends to $0$ as $x$ tends to $\infty$, and
$$
\sum_{k=0}^{\lfloor x\rfloor}\frac{p}{\lfloor x\rfloor}\left(1-\frac{k}{\lfloor x\rfloor}\right)^{n-1}\left(\frac{k}{\lfloor x\rfloor}\right)^{p-1}=\sum_{k=0}^{\lfloor x\rfloor-1}\frac{p}{\lfloor x\rfloor}\left(1-\frac{k}{\lfloor x\rfloor}\right)^{n-1}\left(\frac{k}{\lfloor x\rfloor}\right)^{p-1}
$$
is a Riemann sum of $f$ associated to $\mathcal{P}_x$, we have that
$$
\lim_{x\to\infty}\sum_{k=0}^{\lfloor x\rfloor}\frac{p}{\lfloor x\rfloor}\left(1-\frac{k}{\lfloor x\rfloor}\right)^{n-1}\left(\frac{k}{\lfloor x\rfloor}\right)^{p-1}={{n-1+p}\choose{n-1}}^{-1}.
$$

Moreover, as $\displaystyle{\lim_{x\to\infty}\frac{\lfloor x\rfloor}{x}=1}$, we obtain that
$$
\lim_{x\to\infty} B_{x}(p)=\lim_{x\to\infty}\sum_{k=0}^{\lfloor x\rfloor}\frac{p}{x}\left(1-\frac{k}{x}\right)^{n-1}\left(\frac{k}{x}\right)^{p-1}={{n-1+p}\choose{n-1}}^{-1}.
$$
\end{proof}

From now on we will consider convex bodies satisfying some hypotheses. We will say that a convex body $K$ satisfies the hypotheses \hypertarget{(H)}{(H)} if it satisfies
\begin{enumerate}
\item[a)] $\displaystyle{\label{Hypotheses1}\max_{y\in P_{e_n^\perp}(K)\cap\Z^n}G_1(S_{e_n}(K)\cap(y+\langle e_n\rangle))=G_1(S_{e_n}(K)\cap\langle e_n\rangle)}$,
\item[b)] $\displaystyle{\label{Hypotheses2}M:=\max_{x\in S_{e_n}(K)\cap\Z^n}\langle x,e_n\rangle\geq 1}$.
\end{enumerate}

In the following lemma we prove under the hypotheses \hyperlink{(H)}{(H)}, for any $p\geq1$, the existence of the number $m_0(p)$ that we will need in order to construct the  $\frac{1}{n-1}$-affine defined in \eqref{eq:Definitiong_p}, which will have one crossing point with the function $\tilde{f}$ defined in \eqref{eq:DefinitionfTilde}. Before stating the lemma, let us make the following remark regarding the supports of the functions $f$ and $\tilde{f}$:

\begin{rmk}\label{rmk:Supports}
Let us point out that, since $S_{e_n}(K)\subseteq S_{e_n}(K)+C_{n-1}$ we have that $f(r)\leq\tilde{f}(r)$ for every $r\in[0,\infty)$ and then
$$
\textrm{supp}(f)\subseteq\textrm{supp}(\tilde{f}).
$$

Notice also that, from the definition of $M$, there exists $y\in P_{e_n^\perp}(K)\cap\Z^n$ such that $y+Me_n\in S_{e_n}(K)\cap\Z^n$. Besides, $S_{e_n}(K)\cap\Z^n\cap\{x\in\R^n\,:\,\langle x,e_n\rangle=M+1\}=\emptyset$. Therefore, since $M+1\in\N$, we have that $f(M+1)=0$ and $[0,M]\subseteq\textrm{supp}(f)\subseteq[0, M+1)$. Nevertheless, even though $f(M+1)=0$, it is possible for any integer $k>M$ that $\tilde{f}(k)>0$, as the example given by $K=S_{e_n}(K)=\textrm{conv}\left\{\left(0,\pm1\right),\left(\frac{1}{2},\pm k\right), (1,\pm 1)\right\}\subseteq\R^2$ shows.

Notice also that if $r>\frac{1}{2}\max_{y\in P_{e_n^\perp}}\vol_1(K\cap(y+\langle e_n\rangle))$, since $C_{n-1}\subseteq e_n^\perp$, we have
$$
S_{e_n}(K)\cap\{x\in\R^n\,:\,\langle x,e_n\rangle=r\}=\emptyset
$$
and
$$
(S_{e_n}(K)+C_{n-1})\cap\{x\in\R^n\,:\,\langle x,e_n\rangle=r\}=\emptyset,
$$
and then $f(r)=\tilde{f}(r)=0$. Therefore,
$$
\textrm{supp}(f)\subseteq\textrm{supp}(\tilde{f})\subseteq\left[0,\frac{1}{2}\max_{y\in P_{e_n^\perp}}\vol_1(K\cap(y+\langle e_n\rangle))\right].
$$

\end{rmk}

\begin{lemma}\label{lem:Existencem0}
Let $K\subseteq\R^n$ be a convex body with $0\in P_{e_n^\perp}(K)$ satisfying the hypotheses \hyperlink{(H)}{\textrm{(H)}}:
\begin{enumerate}
\item[a)] $\displaystyle{\max_{y\in P_{e_n^\perp}(K)\cap\Z^n}G_1(S_{e_n}(K)\cap(y+\langle e_n\rangle))=G_1(S_{e_n}(K)\cap\langle e_n\rangle)}$,
\item[b)] $\displaystyle{M:=\max_{x\in S_{e_n}(K)\cap\Z^n}\langle x,e_n\rangle}\geq 1$.
\end{enumerate}
Then, for any $p\geq1$, there exists $m_0(p)\geq M$ such that $m_0(p)>1$ and
$$
m_0(p)^pB_{m_0(p)}(p)=\frac{1}{G_{n-1}(P_{e_n^\perp}(K))}\sum_{k=0}^\infty pk^{p-1}\tilde{f}(k),
$$
where $\tilde{f}:[0,\infty)\to\N\cup\{0\}$ is given by
$$
\tilde{f}(r)=G_{n-1}((S_{e_n}(K)+C_{n-1})\cap\{x\in\R^n\,:\,\langle x,e_n\rangle=r\}).
$$
\end{lemma}

\begin{proof}
Let $K\subseteq\R^n$ be a convex body with $0\in P_{e_n^\perp}(K)$ satisfying the hypotheses \hyperlink{(H)}{(H)} and let  $f:[0,\infty)\to\R$ be the function given by
$$
f(r)=G_{n-1}(S_{e_n}(K)\cap\{x\in\R^n\,:\,\langle x,e_n\rangle=r\}).
$$
By Remark \ref{rmk:Supports}, $f$ satisfies that $f(r)\leq \tilde{f}(r)$ for every $r\geq0$. Besides, since $M\geq 1$, we have that $f(1)\geq 1$. For any $p\geq1$, let also $h_p:(0,\infty)\to[0,\infty)$ be the function given by
\begin{equation}\label{eq:Identityh_p}
h_p(x)=x^pB_x(p)=\sum_{k=0}^{\lfloor x\rfloor}p\left(1-\frac{k}{x}\right)^{n-1}k^{p-1}.
\end{equation}
Our purpose is to prove the existence of some $m_0(p)\geq M$ such that $m_0(p)>1$ and
$$
h_p(m_0(p))=\frac{1}{G_{n-1}(P_{e_n^\perp}(K))}\sum_{k=0}^\infty pk^{p-1}\tilde{f}(k).
$$
Notice that for any $p\geq 1$, $h_p$ is clearly continuous at any $x_0\in(0,\infty)\setminus\N$. If $x_0=k_0\in\N$ then
\begin{eqnarray*}
\lim_{x\to x_0^-}h_p(x)&=&\lim_{x\to k_0^-}\sum_{k=0}^{k_0-1}p\left(1-\frac{k}{x}\right)^{n-1}k^{p-1}=\sum_{k=0}^{k_0-1}p\left(1-\frac{k}{k_0}\right)^{n-1}k^{p-1}\cr
&=&\sum_{k=0}^{k_0}p\left(1-\frac{k}{k_0}\right)^{n-1}k^{p-1}=h_p(k_0)\cr
\end{eqnarray*}
and
\begin{eqnarray*}
\lim_{x\to x_0^+}h_p(x)&=&\lim_{x\to k_0^+}\sum_{k=0}^{k_0}p\left(1-\frac{k}{x}\right)^{n-1}k^{p-1}=\sum_{k=0}^{k_0}p\left(1-\frac{k}{k_0}\right)^{n-1}k^{p-1}\cr
&=&h_p(k_0).\cr
\end{eqnarray*}
Thus, for any $p\geq 1$, $h_p$ is continuous at $x_0=k_0$ and then $h_p$ is continuous on $(0,\infty)$. Besides, since for every  $p\geq1$ the function $h_p$ is defined as $h_p(x)=x^pB_x(p)$ for every $x\in(0,\infty)$, we have that for any $p\geq 1$ Lemma \ref{lem:limit} implies that
\begin{equation}\label{eq:limh_pInfty}
\lim_{x\to\infty}h_p(x)=\infty.
\end{equation}

Let us now distinguish two cases, depending on whether $p>1$ or $p=1$: Assume first that $p>1$. Since $B_x(p)=0$ for every $x\in(0,1]$,  we have that for every $x\in(0,1]$
\begin{equation}\label{eq:limh_p0}
h_p(x)=0.
\end{equation}

Since, by definition, $M$ is a non-negative integer and, by b), we are assuming that $M\geq1$, we have that $\tilde{f}(1)\geq f(1)\geq1>0$ and then
$$
\frac{1}{G_{n-1}(P_{e_n^\perp}(K))}\sum_{k=0}^{\infty}pk^{p-1}\tilde{f}(k)\geq\frac{p\tilde{f}(1)}{G_{n-1}(P_{e_n^\perp}(K))}>0.
$$
By the continuity of $h_p$ on $(0,\infty)$, \eqref{eq:limh_p0} and \eqref{eq:limh_pInfty}, there exists $m_0(p)>1$ such that
$$
h_p(m_0(p))=\frac{1}{G_{n-1}(P_{e_n^\perp}(K))}\sum_{k=0}^{\infty}pk^{p-1}\tilde{f}(k)>0.
$$

Let us assume now that $p=1$. Then $B_x(1)=\frac{1}{x}$ for every $x\in (0,1]$ and, since for every $x\in(0,\infty)$ the function $h_1$ is defined as $h_1(x)=xB_x(1)$, we have that for every $x\in (0,1]$
\begin{equation}\label{eq:limh_10}
h_1(x)=xB_x(1)=1.
\end{equation}

Since by b) we are assuming that $M\geq1$,
$$
\frac{1}{G_{n-1}(P_{e_n^\perp}(K))}\sum_{k=0}^{\infty}pk^{p-1}\tilde{f}(k)=\frac{1}{G_{n-1}(P_{e_n^\perp}(K))}\sum_{k=0}^{\infty}\tilde{f}(k)\geq\frac{\tilde{f}(0)+\tilde{f}(1)}{G_{n-1}(P_{e_n^\perp}(K))}>1,
$$
where we have used that  $\tilde{f}(1)\geq1$ and that, from the definition of $\tilde{f}$,
$$
\tilde{f}(0)=G_{n-1}(P_{e_n^\perp}(K)+C_{n-1})\geq G_{n-1}(P_{e_n^\perp}(K)).
$$
By the continuity of $h_1$ on $(0,\infty)$, \eqref{eq:limh_10} and \eqref{eq:limh_pInfty}, there exists $m_0(1)>1$ such that
$$
h_1(m_0(1))=\frac{1}{G_{n-1}(P_{e_n^\perp}(K))}\sum_{k=0}^{\infty}pk^{p-1}\tilde{f}(k)>0.
$$

In both cases, since $m_0(p)>1$, such $m_0(p)$ satisfies that $B_{m_0(p)}(p)>0$ and
$$
m_0(p)^p=\frac{\sum_{k=0}^{\infty}pk^{p-1}\tilde{f}(k)}{G_{n-1}(P_{e_n^\perp}(K))B_{m_0(p)}(p)}.
$$

Let us now see that $m_0(p)\geq M$. Assume that $m_0(p)<M$. By a) we are assuming
$$
\max_{y\in P_{e_n^\perp}(K)\cap\Z^n}G_1(S_{e_n}(K)\cap(y+\langle e_n\rangle))=G_1(S_{e_n}(K)\cap\langle e_n\rangle).
$$
Thus, from the definition of $\displaystyle{M=\max_{x\in S_{e_n}(K)\cap\Z^n}\langle x,e_n\rangle}$, together with the hypotheses \hyperlink{(H)}{\textrm{(H)}}, we have that $Me_n\in S_{e_n}(K)\cap\{x\in\R^n\,:\langle x,e_n\rangle\geq0\}$. Consequently, the convex hull of $P_{e_n^\perp}(K)$ and the point $m_0(p)e_n$  is contained in $S_{e_n}(K)\cap\{x\in\R^n\,:\langle x,e_n\rangle\geq0\}$. Therefore, for every $0\leq k\leq\lfloor m_0(p)\rfloor$,
$$
S_{e_n}(K)\cap\{x\in\R^n\,:\,\langle x,e_n\rangle=k\}\supseteq\left(1-\frac{k}{m_0(p)}\right)P_{e_n^\perp}(K)\times\{k\}.
$$
Thus,
$$
(S_{e_n}(K)+C_{n-1})\cap\{x\in\R^n\,:\,\langle x,e_n\rangle=k\}\supseteq\left(\left(1-\frac{k}{m_0(p)}\right)P_{e_n}(K)+C_{n-1}\right)\times\{k\}.
$$
We obtain, as a consequence of the discrete Brunn-Minkowski inequality (Theorem \ref{thm: BM_lattice_point_no_G(K)G(L)>0}), that
\begin{eqnarray}\label{eq:InequalityfTildeAndAffine}
\tilde{f}^\frac{1}{n-1}(k)&\geq&\left(1-\frac{k}{m_0(p)}\right)G_{n-1}(P_{e_n^\perp}(K))^\frac{1}{n-1}+\frac{k}{m_0(p)}G_{n-1}(\{0\})^\frac{1}{n-1}\cr
&\geq&\left(1-\frac{k}{m_0(p)}\right)G_{n-1}(P_{e_n^\perp}(K))^\frac{1}{n-1}.
\end{eqnarray}
Since $\lfloor m_0(p)\rfloor\leq m_0(p)<M$ and
$$
\tilde{f}^\frac{1}{n-1}(M)\geq f^\frac{1}{n-1}(M)\geq1>0,
$$
we obtain that
$$
\sum_{k=0}^{M}pk^{p-1}\tilde{f}(k)>\sum_{k=0}^{\lfloor m_0(p)\rfloor}pk^{p-1}\tilde{f}(k)
$$
and then, as a consequence of \eqref{eq:InequalityfTildeAndAffine} and using \eqref{eq:Identityh_p}, that
\begin{eqnarray*}
m_0(p)^p&=&\frac{\sum_{k=0}^{\infty}pk^{p-1}\tilde{f}(k)}{G_{n-1}(P_{e_n^\perp}(K))B_{m_0(p)}(p)}\geq\frac{\sum_{k=0}^{M}pk^{p-1}\tilde{f}(k)}{G_{n-1}(P_{e_n^\perp}(K))B_{m_0(p)}(p)}\cr
&>&\frac{\sum_{k=0}^{\lfloor m_0(p)\rfloor}pk^{p-1}\tilde{f}(k)}{G_{n-1}(P_{e_n^\perp}(K))B_{m_0(p)}(p)}\geq\frac{\sum_{k=0}^{\lfloor m_0(p)\rfloor}pk^{p-1}\left(1-\frac{k}{m_0(p)}\right)^{n-1}}{B_{m_0(p)}(p)}\cr
&=&m_0(p)^p,
\end{eqnarray*}
which is a contradiction. Therefore, $m_0(p)\geq M$.
\end{proof}

\begin{rmk}\label{rmk:m0(p)>1}
Notice that we have obtained in the proof that, in any case, $m_0(p)>1$. Therefore, $B_{m_0(p)}(q)>0$ for every $p\geq1$ and every $q\geq1$.
\end{rmk}

The following lemma shows that the $\frac{1}{n-1}$-affine function $g_p$ defined in \eqref{eq:Definitiong_p}, with $m_0(p)$ the number obtained in Lemma \ref{lem:Existencem0}, has a crossing point with the function $\tilde{f}$ defined in \eqref{eq:DefinitionfTilde}. Such function $g_p$ is constructed so that for every $q>0$
\begin{equation}\label{eq:Sumg_p}
\sum_{k=0}^\infty qk^{q-1}g_p(k)=m_0(p)^qB_{m_0(p)}(q)G_{n-1}(P_{e_n^\perp}(K)),
\end{equation}
with $m_0(p)$ provided by the previous lemma satisfying that
\begin{equation}\label{eq:m_0(p)}
\sum_{k=0}^{\infty} pk^{p-1}\tilde{f}(k)=m_0(p)^pB_{m_0(p)}(p)G_{n-1}(P_{e_n^\perp}(K)).
\end{equation}

\begin{lemma}\label{lem:existenceDiscreter0}
Let $K\subseteq\R^n$ be a convex body with $0\in P_{e_n^\perp}(K)$ satisfying the hypotheses \hyperlink{(H)}{\textrm{(H)}} and, for any $p\geq 1$, let $m_0(p)$ be given by Lemma \ref{lem:Existencem0}. Let $f,\tilde{f}:[0,\infty)\to\N\cup\{0\}$ be the functions given by
\begin{itemize}
\item $f(r)=G_{n-1}(S_{e_n}(K)\cap\{x\in\R^n\,:\,\langle x,e_n\rangle=r\})$,
\item $\tilde{f}(r)=G_{n-1}((S_{e_n}(K)+C_{n-1})\cap\{x\in\R^n\,:\,\langle x,e_n\rangle=r\})$
\end{itemize}
and let $g_p:[0,\infty)\to[0,\infty)$ be the function defined in \eqref{eq:Definitiong_p}, given by
$$
g_p(x)=\begin{cases}
\left(1-\frac{x}{m_0(p)}\right)^{n-1}G_{n-1}(P_{e_n^\perp}(K)) &\textrm{ if }0\leq x\leq m_0(p)\cr
0&\textrm{ if }x>m_0(p).\cr
\end{cases}
$$
Then, there exists $r_0(p)\in[0,\infty)$ such that
$$
\tilde{f}(k)\geq g_p(k)
$$
for every $k\in\N\cup\{0\}$ with $0\leq k< r_0(p)$ and
$$
g_p(k)\geq f(k)
$$
for every $k\in\N$ with $k\geq r_0(p)$.
\end{lemma}

\begin{proof}
Let $K\subseteq\R^n$ be a convex body with $0\in P_{e_n^\perp}(K)$ satisfying the hypotheses \hyperlink{(H)}{(H)} and let $f,\tilde{f}$, and $g_p$ be the functions defined in the statement. Let us denote, for every $r\geq0$,
$$
D_r=S_{e_n}(K)\cap\{x\in\R^n\,:\,\langle x,e_n\rangle=r\},
$$
and
$$
\tilde{D}_r=(S_{e_n}(K)+C_{n-1})\cap\{x\in\R^n\,:\,\langle x,e_n\rangle=r\}=D_r+C_{n-1}.
$$
Therefore, $f(r)=G_{n-1}(D_r)$ and $\tilde{f}(r)=G_{n-1}(\tilde{D}_r)$ for every $r\geq0$.

Notice that, as mentioned in Remark \ref{rmk:Supports},
$$
\textrm{supp}(f)\subseteq\textrm{supp}(\tilde{f})\subseteq\left[0,\frac{1}{2}\max_{y\in P_{e_n^\perp}}\vol_1(K\cap(y+\langle e_n\rangle))\right]
$$
and then, for every $\displaystyle{r>\max_{y\in P_{e_n^\perp}(K)}\frac{1}{2}\vol_1(K\cap(y+\langle e_n\rangle))}$, we have
$$
0=\tilde{f}(r)\leq g_p(r).
$$
Therefore, we can define
$$
\tilde{r}_0(p)=\inf\{r\geq0\,:\,\tilde{f}(r)\leq g_p(r)\}\in[0,\infty).
$$
From the definition of $\tilde{r}_0(p)$, we trivially have that for every $0\leq r<\tilde{r}_0(p)$
$$
\tilde{f}(r)> g_p(r).
$$
In particular, for every $r=k\in\N\cup\{0\}$ with $0\leq k<\tilde{r}_0(p)$, we have that
$$
\tilde{f}(k)>g_p(k).
$$

Notice also that $M$ is defined in \hyperlink{(H)}{(H)} as $\displaystyle{M=\max_{x\in S_{e_n}(K)\cap\Z^n}\langle x, e_n\rangle}$ and, as mentioned in Remark \ref{rmk:Supports}, $\textrm{supp}(f)\subseteq[0, M+1)$. Then, for every $k\in\N$ such that $k>M$ we have
$$
g_p(k)\geq f(k)=0.
$$
Therefore, if $\tilde{r}_0(p)>M$, then we can take $r_0(p)=\tilde{r}_0(p)$. Let us assume that $\tilde{r}_0(p)\leq M$.

Let us assume first that $\tilde{r}_0(p)=M\in\N$. If $f(M)\leq g_p(M)$, then we can take $r_0(p)=\tilde{r}_0(p)$. If, on the contrary, $f(M)>g_p(M)$, we have, by Remark \ref{rmk:Supports}
$$
\tilde{f}(M)\geq f(M)>g_p(M)
$$
and we can take any $r_0(p)>\tilde{r}_0(p)=M$.

Let us assume now that $\tilde{r}_0(p)<M$. In such case, let us take $r_0(p)=\tilde{r}_0(p)$ and let us see that for every $k\in\N$ with $r_0(p)\leq k\leq M$ we have that $g_p(k)\geq f(k)$. Notice that for every $0\leq r\leq M$, the set $\tilde{D}_r$ is a convex open (in the topology induced in $\R^{n-1}\times\{r\}$ by the standard topology in $\R^n$) bounded set, and that $P_{e_n^\perp}(\tilde{D}_{r_1})\supseteq P_{e_n^\perp}(\tilde{D}_{r_2})$ if $r_1\leq r_2$. Moreover, for any $0\leq r_1<\max_{y\in P_{e_n^\perp}(K)}\frac{1}{2}\vol_1(K\cap(y+\langle e_n\rangle))$ we have $\displaystyle{\bigcup_{r_2>r_1}P_{e_n^\perp}(\tilde{D}_{r_2})=P_{e_n^\perp}(\tilde{D}_{r_1})}$. Therefore, the function $\tilde{f}$ is continuous from the right at every $0\leq r_1<\max_{y\in P_{e_n^\perp}(K)}\frac{1}{2}\vol_1(K\cap(y+\langle e_n\rangle))$ and, from the definition of $\tilde{r}_0(p)$ as an infimum, we obtain that if $r_0(p)=\tilde{r}_0(p)<\max_{y\in P_{e_n^\perp}(K)}\frac{1}{2}\vol_1(K\cap(y+\langle e_n\rangle))$, then
\begin{equation}\label{eq:fTildeAndg}
\tilde{f}(r_0(p))\leq g_p(r_0(p))
\end{equation}
and
$$
r_0(p)=\min\{r\geq0\,:\,\tilde{f}(r)\leq g_p(r)\}.
$$

Notice that, as explained in Remark \ref{rmk:Supports},
$$
[0,M]\subseteq\textrm{supp}(f)\subseteq\textrm{supp}(\tilde{f})\subseteq\left[0,\frac{1}{2}\max_{y\in P_{e_n^\perp}}\vol_1(K\cap(y+\langle e_n\rangle))\right].
$$
Thus, $M\leq \max_{y\in P_{e_n^\perp}(K)}\frac{1}{2}\vol_1(K\cap(y+\langle e_n\rangle))$ and, since $r_0(p)<M$, then $r_0(p)$ satisfies \eqref{eq:fTildeAndg}. If $r_0(p)< r\leq M$, then, taking $\lambda=\frac{r_0(p)}{r}\in[0,1)$, we have, by the convexity of $S_{e_n}(K)+C_{n-1}$, that
$$
\tilde{D}_{r_0(p)}\supseteq (1-\lambda)\tilde{D}_0+\lambda \tilde{D}_r=(1-\lambda)D_0+\lambda D_r+C_{n-1}.
$$
By the discrete Brunn-Minkowski inequality (Theorem \ref{thm: BM_lattice_point_no_G(K)G(L)>0}), we have
\begin{eqnarray*}
\tilde{f}^\frac{1}{n-1}(r_0(p))&\geq&G_{n-1}((1-\lambda)D_0+\lambda D_r+C_{n-1})^\frac{1}{n-1}\cr
&\geq&(1-\lambda)f^\frac{1}{n-1}(0)+\lambda f^\frac{1}{n-1}(r).
\end{eqnarray*}
Taking into account that $r_0(p)<r\leq M\leq m_0(p)$,
\begin{eqnarray*}
g_p^\frac{1}{n-1}(r_0(p))&=&\left(1-\frac{r_0(p)}{m_0(p)}\right)G_{n-1}(P_{e_n^\perp}(K))^\frac{1}{n-1}\cr
&=&(1-\lambda)G_{n-1}(P_{e_n^\perp}(K))^\frac{1}{n-1}+\lambda\left(1-\frac{r}{m_0}\right)G_{n-1}(P_{e_n^\perp}(K))^\frac{1}{n-1}\cr
&=&(1-\lambda)f^\frac{1}{n-1}(0)+\lambda g_p^\frac{1}{n-1}(r).\cr
\end{eqnarray*}
Thus, by \eqref{eq:fTildeAndg}, we have that for every $r_0(p)<r\leq M$
$$
(1-\lambda)f^\frac{1}{n-1}(0)+\lambda g_p^\frac{1}{n-1}(r)=g_p^\frac{1}{n-1}(r_0(p))\geq \tilde{f}^\frac{1}{n-1}(r_0(p))\geq (1-\lambda)f^\frac{1}{n-1}(0)+\lambda f^\frac{1}{n-1}(r)
$$
Therefore, for every $r_0(p)<r\leq M$,
$$
g_p(r)\geq f(r).
$$
By \eqref{eq:fTildeAndg}, this inequality also holds for $r=r_0(p)$. Therefore, if $\tilde{r}_0(p)=r_0(p)<M$, for every $k\in\N$ such that $r_0(p)\leq k\leq M$, we have that
$$
g_p(k)\geq f(k).
$$
\end{proof}

We are now able to prove the following theorem, which will give Theorem \ref{thm:PurelyDiscreteZhang}.

\begin{thm}\label{thm:CompletelyDiscreteBerwald}
Let $K\subseteq\R^n$ be a convex body with $0\in P_{e_n^\perp}(K)$ satisfying the hypotheses \hyperlink{(H)}{\textrm{(H)}}. Let $p\geq1$ and $m_0(p)$ be given by Lemma \ref{lem:Existencem0}. Then, for any $p<q$ we have that
$$
\left(\frac{B_{m_0(p)}(q)^{-1}}{G_{n-1}(P_{e_n^\perp}(K))}\sum_{k=0}^{\infty}qk^{q-1}f(k)\right)^\frac{1}{q}\leq \left(\frac{B_{m_0(p)}(p)^{-1}}{G_{n-1}(P_{e_n^\perp}(K))}\sum_{k=0}^{\infty}pk^{p-1}\tilde{f}(k)\right)^\frac{1}{p},
$$
where the functions $f,\tilde{f}:[0,\infty)\to\N\cup\{0\}$ are given by
\begin{itemize}
\item $f(r)=G_{n-1}(S_{e_n}(K)\cap\{x\in\R^n\,:\,\langle x,e_n\rangle=r\})$,
\item $\tilde{f}(r)=G_{n-1}((S_{e_n}(K)+C_{n-1})\cap\{x\in\R^n\,:\,\langle x,e_n\rangle=r\})$.
\end{itemize}
\end{thm}

\begin{proof}
Let $K\subseteq\R^n$ be a convex body with $0\in P_{e_n^\perp}(K)$  satisfying the hypotheses \hyperlink{(H)}{(H)} and let $p\geq 1$. By Lemma \ref{lem:Existencem0} and Remark \ref{rmk:m0(p)>1}, there exists $m_0(p)>1$ such that
$$
m_0(p)^pB_{m_0(p)}(p)=\frac{1}{G_{n-1}(P_{e_n^\perp}(K))}\sum_{k=0}^\infty pk^{p-1}\tilde{f}(k)>0.
$$
Let also $g_p$  be the function defined in \eqref{eq:Definitiong_p}, given by
$$
g_p(r)=\begin{cases}
\left(1-\frac{r}{m_0(p)}\right)^{n-1}G_{n-1}(P_{e_n^\perp}(K)) &\textrm{ if }0\leq r\leq m_0(p)\cr
0&\textrm{ if }r>m_0(p),\cr
\end{cases}
$$
which, from the definition of $B_{m_0(p)}(q)$ for $q>0$, satisfies \eqref{eq:Sumg_p}. That is, for every $q>0$
$$
\sum_{k=0}^\infty qk^{q-1}g_p(k)=m_0(p)^qB_{m_0(p)}(q)G_{n-1}(P_{e_n^\perp}(K)).
$$
Therefore, taking $q=p$
$$
\sum_{k=0}^{\infty}k^{p-1}g_p(k)=\sum_{k=0}^{\lfloor m_0(p)\rfloor}k^{p-1}g_p(k)=\sum_{k=0}^\infty k^{p-1}\tilde{f}(k).
$$

Let now $q>p$. By Lemma \ref{lem:existenceDiscreter0} there exists $r_0(p)$ such that $\tilde{f}(k)\geq g_p(k)$
for every $k\in\N\cup\{0\}$ with $0\leq k< r_0(p)$ and $g_p(k)\geq f(k)$
for every $k\in\N$ with $k\geq r_0(p)$. Taking into account that $f(k)\leq\tilde{f}(k)$ for every $k\geq0$, and understanding the first sum as $0$ if $r_0(p)=0$, we have
\begin{eqnarray*}
&&\sum_{k=0}^{\lceil r_0(p)\rceil-1}k^{q-1}\left(\tilde{f}(k)-g_p(k)\right)-\sum_{k=\lceil r_0(p)\rceil}^\infty k^{q-1}\left(g_p(k)-f(k)\right)\cr
&=&\sum_{k=0}^{\lceil r_0(p)\rceil-1}k^{q-p}k^{p-1}\left(\tilde{f}(k)-g_p(k)\right)-\sum_{k=\lceil r_0(p)\rceil}^\infty k^{q-p}k^{p-1}\left(g_p(k)-f(k)\right)\cr
&\leq&{r_0(p)}^{q-p}\sum_{k=0}^{\lceil r_0(p)\rceil-1}k^{p-1}\left(\tilde{f}(k)-g_p(k)\right)-{r_0(p)}^{q-p}\sum_{k=\lceil r_0(p)\rceil}^\infty k^{p-1}\left(g_p(k)-f(k)\right)\cr
&=&{r_0(p)}^{q-p}\sum_{k=0}^{\lceil r_0(p)\rceil-1}k^{p-1}\left(\tilde{f}(k)-g_p(k)\right)+{r_0(p)}^{q-p}\sum_{k=\lceil r_0(p)\rceil}^\infty k^{p-1}\left(f(k)-g_p(k)\right)\cr
&\leq&{r_0(p)}^{q-p}\sum_{k=0}^{\lceil r_0(p)\rceil-1}k^{p-1}\left(\tilde{f}(k)-g_p(k)\right)+{r_0(p)}^{q-p}\sum_{k=\lceil r_0(p)\rceil}^\infty k^{p-1}\left(\tilde{f}(k)-g_p(k)\right)\cr
&=&{r_0}^{q-p}\sum_{k=0}^{\infty}k^{p-1}\left(\tilde{f}(k)-g_p(k)\right)=0.
\end{eqnarray*}
Therefore, we have that
$$
\sum_{k=0}^{\lceil r_0(p)\rceil-1}k^{q-1}\tilde{f}(k)+\sum_{k=\lceil r_0(p)\rceil}^\infty k^{q-1}f(k)\leq\sum_{k=0}^\infty k^{q-1}g_p(k).
$$
and then, since $\tilde{f}(k)\geq f(k)$ for every $k\geq0$,
$$
\sum_{k=0}^{\infty}k^{q-1}f(k)\leq\sum_{k=0}^\infty k^{q-1}g_p(k).
$$
Consequently, since $m_0(p)$  satisfies \eqref{eq:m_0(p)}, $g_p$ satisfies \eqref{eq:Sumg_p}, and $B_{m_0(p)}(q)>0$ for every $q\geq 1$, as mentioned in Remark \ref{rmk:m0(p)>1},
\begin{eqnarray*}
\left(\frac{B_{m_0(p)}(q)^{-1}}{G_{n-1}(P_{e_n^\perp}(K))}\sum_{k=0}^{\infty}qk^{q-1}f(k)\right)^\frac{1}{q}&\leq&\left(\frac{B_{m_0(p)}(q)^{-1}}{G_{n-1}(P_{e_n^\perp}(K))}\sum_{k=0}^{\infty}qk^{q-1}g_p(k)\right)^\frac{1}{q}=m_0(p)\cr
&=&\left(\frac{B_{m_0(p)}(p)^{-1}\sum_{k=0}^\infty pk^{p-1}\tilde{f}(k)}{G_{n-1}(P_{e_n^\perp}(K))}\right)^\frac{1}{p}.\cr
\end{eqnarray*}
\end{proof}

We can finally prove Theorem \ref{thm:PurelyDiscreteZhang}:

\begin{proof}[Proof of Theorem \ref{thm:PurelyDiscreteZhang}]
Let $K\subseteq\R^n$ be a convex body with $0\in P_{e_n^\perp}(K)$ such that $\displaystyle{\max_{y\in P_{e_n^\perp}(K)\cap\Z^n}G_1(S_{e_n}(K)\cap(y+\langle e_n\rangle))=G_1(S_{e_n}(K)\cap\langle e_n\rangle)}$ and let\\ $\displaystyle{M=\max_{x\in S_{e_n}(K)\cap\Z^n}\langle x,e_n\rangle}$. If $M=0$ then, as mentioned in Remark \ref{rmk:RemarkM=0}, any value of $m_0>1$ gives the inequality, even though $m_0$ is not defined by Lemma \ref{lem:Existencem0}. Let us assume that $M\geq 1$. Therefore $K$ satisfies the hypotheses \hyperlink{(H)}{(H)}. Let $p=1$, $f,\tilde{f}:[0,\infty)\to\N\cup\{0\}$ be defined as in Theorem \ref{thm:CompletelyDiscreteBerwald} by
\begin{itemize}
\item $f(r)=G_{n-1}(S_{e_n}(K)\cap\{x\in\R^n\,:\,\langle x,e_n\rangle=r\})$,
\item $\tilde{f}(r)=G_{n-1}((S_{e_n}(K)+C_{n-1})\cap\{x\in\R^n\,:\,\langle x,e_n\rangle=r\})$,
\end{itemize}
and let $m_0=m_0(1)>1$ be the number given by Lemma \ref{lem:Existencem0}. Applying Theorem \ref{thm:CompletelyDiscreteBerwald} with $p=1$ and $q=n+1$ we obtain that
$$
\left(\frac{B_{m_0}(n+1)^{-1}}{G_{n-1}(P_{e_n^\perp}(K))}\sum_{k=0}^{\infty}(n+1)k^{n}f(k)\right)^\frac{1}{n+1}\leq \frac{B_{m_0}(1)^{-1}}{G_{n-1}(P_{e_n^\perp}(K))}\sum_{k=0}^{\infty}\tilde{f}(k).
$$
Equivalently,
$$
B_{m_0}(n+1)^{-1}\sum_{k=0}^{\infty}(n+1)k^{n}f(k)\leq \frac{\left(B_{m_0}(1)^{-1}\sum_{k=0}^{\infty}\tilde{f}(k)\right)^{n+1}}{(G_{n-1}(P_{e_n^\perp}(K)))^n}.
$$
From \eqref{eq:SteinerSymmetrization2}, $S_{e_n}(K)$ is symmetric respect to the hyperplane $\{x\in\R^n\,:\,\langle x, e_n\rangle=0\}$ and then, from the definition of $f$ and taking into account that for any $y\in\Z^{n-1}$ and $k\in\N$ we have $(y,k)\neq(y,-k)$ and that the term corresponding to $k=0$ in the following sum is $0$, we obtain
$$
\sum_{k=0}^{\infty}k^{n}f(k)=\frac{1}{2}\sum_{x\in S_{e_n}(K)\cap\Z^n}|\langle x,e_n\rangle|^n=\frac{1}{2}\int_{S_{e_n}(K)}|\langle x,e_n\rangle|^ndG_n(x).
$$
In the same way, since $S_{e_n}(K)+C_{n-1}$ is also symmetric respect to the hyperplane $\{x\in\R^n\,:\,\langle x, e_n\rangle=0\}$ and taking into account that for any $y\in\Z^{n-1}$ and $k\in\N$ we have $(y,k)\neq(y,-k)$, but $(y,0)=(y,-0)$ and the term corresponding to the following sum is not 0, we obtain
$$
\sum_{k=0}^{\infty}\tilde{f}(k)=\frac{1}{2}\left(G_n(S_{e_n}(K)+C_{n-1})+G_{n-1}(P_{e_n^\perp}(K)+C_{n-1})\right).
$$
Therefore,
\begin{eqnarray*}
&&\frac{(n+1)B_{m_0}(n+1)^{-1}}{B_{m_0}(1)^{-(n+1)}}2^n\int_{S_{e_n}(K)}|\langle x,e_n\rangle|^ndG_n(x)\leq\cr
&\leq&\frac{\left(G_n(S_{e_n}(K)+C_{n-1})+G_{n-1}(P_{e_n^\perp}(K)+C_{n-1})\right)^{n+1}}{(G_{n-1}(P_{e_n^\perp}(K)))^n}.
\end{eqnarray*}
\end{proof}

Finally, let us see that Theorem \ref{thm:PurelyDiscreteZhang} implies Theorem \ref{thm:ZhangPreIntegration2}. Let us recall that, by Lemma \ref{lem:IdentitiesIntegral}, Theorem \ref{thm:ZhangPreIntegration2} is an equivalent form of Theorem \ref{thm:ZhangPreIntegration}, which implies, as seen in Corollary \ref{cor:Zhang}, Zhang's inequality \eqref{eq:Zhang}.

\begin{cor}[Theorem \ref{thm:ZhangPreIntegration2}]\label{prop:FromPurelyDiscreteZhangDiscreteZhangToZhang}
Let $K\subseteq\R^n$ be a convex body. Then
$$
\frac{{{2n}\choose{n}}}{n^n}2^n\int_{S_{e_n}(K)}|\langle x,e_n\rangle|^ndx\leq\frac{(\vol_n(K))^{n+1}}{(\vol_{n-1}(P_{e_n^\perp}(K)))^n}.
$$
\end{cor}

\begin{proof}
Let $K\subseteq\R^n$ be a convex body. Since the inequality we want to prove is invariant under translations we can assume, without loss of generality, that
$$
\max_{y\in P_{e_n^\perp}(K)}\vol_1(K\cap(y+\langle e_n\rangle))=\vol_1(K\cap\langle e_n\rangle).
$$
Therefore, for any $\lambda>0$
$$
\max_{y\in P_{e_n^\perp}(\lambda K)}\vol_1(\lambda K\cap(y+\langle e_n\rangle))=\vol_1(\lambda K\cap\langle e_n\rangle).
$$
Since for every $y\in P_{e_n^\perp}(\lambda K)$ the segment $S_{e_n}(\lambda K)\cap(y+\langle e_n\rangle)$ is centered at $y$, we have
$$
\max_{y\in P_{e_n^\perp}(\lambda K)}G_1(S_{e_n}(\lambda K)\cap(y+\langle e_n\rangle))=G_1(S_{e_n}(\lambda K)\cap\langle e_n\rangle).
$$
In particular,
$$
\max_{y\in P_{e_n^\perp}(\lambda K)\cap\Z^n}G_1(S_{e_n}(\lambda K)\cap(y+\langle e_n\rangle))=G_1(S_{e_n}(\lambda K)\cap\langle e_n\rangle).
$$
By Theorem \ref{thm:PurelyDiscreteZhang}, there exists $\displaystyle{m_{0,\lambda}\geq \max_{x\in S_{e_n}(\lambda K)\cap\Z^n}\langle x,e_n\rangle}$ such that
\begin{eqnarray*}
&&\frac{(n+1)B_{m_{0,\lambda}}(n+1)}{B_{m_{0,\lambda}}(1)^{-(n+1)}}2^n\int_{S_{e_n}(\lambda K)}|\langle x,e_n\rangle|^ndG_n(x)\leq\cr
&\leq&\frac{\left(G_n(S_{e_n}(\lambda K)+C_{n-1})+G_{n-1}(P_{e_n^\perp}(\lambda K)+C_{n-1})\right)^{n+1}}{(G_{n-1}(P_{e_n^\perp}(\lambda K)))^n}.
\end{eqnarray*}
Therefore, taking into account that $S_{e_n}(\lambda K)=\lambda S_{e_n}(K)$ for any $\lambda>0$ and dividing both sides of the inequality by $\lambda^{2n}=\frac{\lambda^{n(n+1)}}{\lambda^{n(n-1)}}$, we have that for every $\lambda>0$
\begin{eqnarray*}
&&\frac{(n+1)B_{m_{0,\lambda}}(n+1)}{B_{m_{0,\lambda}}(1)^{-(n+1)}}2^n\frac{1}{\lambda^n}\int_{\lambda S_{e_n}(K)}|\langle \frac{x}{\lambda},e_n\rangle|^ndG_n(x)\leq\cr
&\leq&\frac{\left(\frac{G_n(\lambda S_{e_n}(K)+C_{n-1})}{\lambda^n}+\frac{1}{\lambda}\frac{G_{n-1}(P_{e_n^\perp}(\lambda K)+C_{n-1})}{\lambda^{n-1}}\right)^{n+1}}{\left(\frac{G_{n-1}(\lambda P_{e_n^\perp}(K))}{\lambda^{n-1}}\right)^n}.
\end{eqnarray*}
Taking the limit as $\lambda\to\infty$  and taking into account that $\displaystyle{\lim_{\lambda\to\infty}\max_{x\in S_{e_n}(\lambda K)\cap\Z^n}\langle x,e_n\rangle=\infty}$, and therefore $\displaystyle{\lim_{\lambda\to\infty}m_{0,\lambda}=\infty}$, we obtain, using Lemma \ref{lem:limit},
$$
\frac{(n+1){{2n}\choose{n-1}}}{n^{n+1}}2^n\int_{S_{e_n}(K)}|\langle x,e_n\rangle|^ndx\leq\frac{(\vol_n(S_{e_n}(K))+0)^{n+1}}{(\vol_{n-1}(P_{e_n^\perp}(K)))^n}=\frac{(\vol_n(K))^{n+1}}{(\vol_{n-1}(P_{e_n^\perp}(K)))^n},
$$
where we have also used \eqref{eq:ApproachingVolByDilationsWithG} and \eqref{eq:ApproachingIntegralWithG}, since the function $|\langle x,e_n\rangle|^n$ is Riemann-integrable on $S_{e_n}(K)$.

Equivalently,
$$
\frac{{{2n}\choose{n}}}{n^n}2^n\int_{S_{e_n}(K)}|\langle x,e_n\rangle|^ndx\leq\frac{(\vol_n(K))^{n+1}}{(\vol_{n-1}(P_{e_n^\perp}(K)))^n}.
$$
\end{proof}
\section{Ball bodies of the discrete covariogram}\label{sec:BallBodies}

In this section we initiate, for any convex body $K\subseteq\R^n$ with $0\in K$, the study of the $p$-th ball bodies of the discrete covariogram function $\tilde{g}_K:\R^n\to[0,\infty)$ given by
\begin{equation}\label{eq:tilegAndCharacteristicFunctions}
\tilde{g}_K(x)=G_n(K\cap(x+K))=\sum_{y\in K\cap\Z^n}\chi_{y-K}(x).
\end{equation}
Since $y-K$ is measurable for every $y\in\R^n$, as it is a compact set, $\tilde{g}_K$ is a measurable function. Besides, it satisfies that $\tilde{g}_K(0)=G_n(K)>0$, since $0\in K$. Therefore, we can consider, for $p>0$, the $p$-th Ball bodies of $\tilde{g_K}$, which are defined by \eqref{eq:Definition p-BallBodies} and are given by
\begin{eqnarray*}
K_p(\tilde{g}_K):&=&\left\{x\in\R^n\,:\,p\int_0^\infty r^{p-1}\tilde{g}_K(rx)dr\geq \tilde{g}_K(0)\right\}\cr
&=&\left\{x\in\R^n\,:\,p\int_0^\infty r^{p-1}G_n(K\cap(rx+K))dr\geq G_n(K)\right\}.
\end{eqnarray*}
The $p$-th ball bodies of $\tilde{g}_K$, $K_p(\tilde{g}_K)$, are not necessarily convex bodies. Nevertheless, they are star sets whose radial function, by \eqref{eq:RadialFunctionp-BallBodies}, is given, for any $\theta\in S^{n-1}$, by
$$
\rho_{K_p(\tilde{g}_K)}^p(\theta)=\frac{p}{G_n(K)}\int_0^\infty r^{p-1}G_n(K\cap(r\theta+K))dr.
$$

\begin{rmk}\label{rmk:OnePoint}
Let us point out that if $K\cap\Z^n=\{0\}$, then, by \eqref{eq:tilegAndCharacteristicFunctions}, $\tilde{g}_K=\chi_{-K}$ and then, by \eqref{eq:BallBodiesCharacteristicFunctions}, $K_p(\tilde{g}_K)=-K$ for every $p>0$.
\end{rmk}

We will also consider the discrete covariogram function of the open set $K+C_n$, $\tilde{g}_{K+C_n}:\R^n\to[0,\infty)$, given by
$$
\tilde{g}_{K+C_n}(x)=G_n((K+C_n)\cap(x+K+C_n))=\sum_{y\in (K+C_n)\cap\Z^n}\chi_{y-(K+C_n)}(x),
$$
which is also measurable since $y-(K+C_n)$ is measurable for every $y\in\R^n$, as $K+C_n$ is an open set, and satisfies that $\tilde{g}_{K+C_n}(0)>0$, since $0\in K\subseteq K+C_n$. The $p$-th Ball bodies of $\tilde{g}_{K+C_n}$ are given by
\begin{eqnarray*}
&&K_p(\tilde{g}_{K+C_n}):=\left\{x\in\R^n\,:\,p\int_0^\infty r^{p-1}\tilde{g}_{K+C_n}(rx)dr\geq \tilde{g}_{K+C_n}(0)\right\}\cr
&=&\left\{x\in\R^n\,:\,p\int_0^\infty r^{p-1}G_n((K+C_n)\cap(rx+K+C_n))dr\geq G_n(K+C_n)\right\},
\end{eqnarray*}
which are star sets with radial function given, for every $\theta\in S^{n-1}$, by
$$
\rho_{K_p(\tilde{g}_{K+C_n})}^p(\theta)=\frac{p}{G_n(K+C_n)}\int_0^\infty r^{p-1}G_n((K+C_n)\cap(r\theta+K+C_n))dr.
$$

Notice that, since $K\subseteq K+C_n$, then $\tilde{g}_K(x)\leq \tilde{g}_{K+C_n}(x)$ for every $x\in\R^n$. Thus, for every $\theta\in S^{n-1}$, we have
$$
G_n(K)\rho_{K_p(\tilde{g}_K)}^p(\theta)\leq G_n(K+C_n)\rho_{K_p(\tilde{g}_{K+C_n})}^p(\theta)
$$
and then
\begin{equation}\label{eq:inclusionBallBodiesDiscreteCovariogram}
K_p(\tilde{g}_K)\subseteq\left(\frac{G_n(K+C_n)}{G_n(K)}\right)^\frac{1}{p}K_p(\tilde{g}_{K+C_n}).
\end{equation}
The purpose of this section is to prove, on the one hand, that even though $K_p(\tilde{g}_K)$ is not necessarily convex, its convex hull is contained in the same dilation of $K_p(\tilde{g}_{K+C_n})$ that appears in \eqref{eq:inclusionBallBodiesDiscreteCovariogram} and, on the other hand, that even though we do not know whether an inclusion relation such as the one given by \eqref{eq:inclusionBallBodiesCovariogram} holds for the $p$-th Ball bodies of the discrete covariogram, a similar inclusion relation holds when substituting $K_p(\tilde{g}_K)$ by the dilation of $K_p(\tilde{g}_{K+C_n})$ given by the right hand side of \eqref{eq:inclusionBallBodiesDiscreteCovariogram}. More precisely, we will prove the following:
\begin{thm}\label{thm:InclusionBallBodiesDiscreteCovariogram}
Let $K\subseteq\R^n$ be a convex body such that $0\in K$. For any $0<p<q$ we have that
$$
{{n+q}\choose{n}}^\frac{1}{q}K_q(\tilde{g}_{K})\subseteq {{n+p}\choose{n}}^\frac{1}{p}\left(\frac{G_n(K+C_n)}{G_n(K)}\right)^\frac{1}{p}K_p(\tilde{g}_{K+C_n}).
$$
\end{thm}

Let us start by showing that the convex hull of $K_p(\tilde{g}_K)$ is contained in in the same dilation of $K_p(\tilde{g}_{K+C_n})$ that appears in \eqref{eq:inclusionBallBodiesDiscreteCovariogram}. The proof will follow the lines of the proof of the convexity of the $p$-th Ball bodies of log-concave function (see \cite[Theorem 2.5.5]{BGVV}), relying on \cite[Theorem 2.5.4]{BGVV}, which gives a lower bound for the integral of a function $h$ on $[0,\infty)$ in terms of some mean of the integrals of two functions $w$ and $g$ on $[0,\infty)$, provided that for any $r,s>0$ the function $h$, evaluated at the same mean of $r$ and $s$ is bounded below by some geometric mean of $w(r)$ and $g(s)$. However, some modifications in the proof will be induced by the necessity of adding the open cube $C_n$ to the set in the left-hand side in the discrete Brunn-Minkowski inequality (Theorem \ref{thm: BM_lattice_point_no_G(K)G(L)>0}).

\begin{proposition}\label{prop:ConvexityBallBodies}
Let $K\subseteq\R^n$ be a convex body with $0\in K$. Then for any $p>0$
$$
\textrm{conv}(K_p(\tilde{g}_K))\subseteq\left(\frac{G_n(K+C_n)}{G_n(K)}\right)^\frac{1}{p}K_p(\tilde{g}_{K+C_n}).
$$
\end{proposition}
\begin{proof}
First of all, let us define $\tilde{K}_p(\tilde{g}_{K+C_n})$ as the following set:
\begin{eqnarray*}
&&\tilde{K}_p(\tilde{g}_{K+C_n})=\left\{x\in\R^n\,:\,p\int_0^\infty r^{p-1}\tilde{g}_{K+C_n}(rx)dr\geq \tilde{g}_K(0)\right\}\cr
&=&\left\{x\in\R^n\,:\,p\int_0^\infty r^{p-1}G_n((K+C_n)\cap(rx+K+C_n))dr\geq G_n(K)\right\}.
\end{eqnarray*}
Notice that since the right-hand side in the inequality defining this set is $\tilde{g}_K(0)$ rather than $\tilde{g}_{K+C_n}(0)$, this set is not $K_p(\tilde{g}_{K+C_n})$ but a dilation of it. Indeed, for any $x\in\R^n$ and any $\lambda>0$
$$
p\int_0^\infty r^{p-1}\tilde{g}_{K+C_n}(\lambda r x)dr=\frac{p}{\lambda^p}\int_0^\infty s^{p-1}\tilde{g}_{K+C_n}(sx)ds
$$
and then, if $0<\lambda\leq 1$,
$$
p\int_0^\infty r^{p-1}\tilde{g}_{K+C_n}(\lambda r x)dr\geq p\int_0^\infty r^{p-1}\tilde{g}_{K+C_n}(rx)dr.
$$
Thus, if $x\in \tilde{K}_p(\tilde{g}_{K+C_n})$, we have that also $\lambda x\in \tilde{K}_p(\tilde{g}_{K+C_n})$. Therefore $\tilde{K}_p(\tilde{g}_{K+C_n})$ is a star set with $0$ as a center. Moreover, as in \eqref{eq:RadialFunctionp-BallBodies}, for any $\theta\in S^{n-1}$
\begin{eqnarray*}
\rho_{\tilde{K}_p(\tilde{g}_{K+C_n})}(\theta)&=&\sup\left\{\lambda>0\,:\,p\int_0^\infty r^{p-1}\tilde{g}_{K+C_n}(\lambda r\theta)dr\geq \tilde{g}_K(0)\right\}\cr
&=&\sup\left\{\lambda>0\,:\,\frac{p}{\lambda^p}\int_0^\infty s^{p-1}\tilde{g}_{K+C_n}(s\theta)ds\geq \tilde{g}_K(0)\right\}\cr
&=&\left(\frac{p}{\tilde{g}_K(0)}\int_0^\infty s^{p-1}\tilde{g}_{K+C_n}(s\theta)ds\right)^\frac{1}{p}\cr
&=&\left(\frac{p}{G_n(K)}\int_0^\infty r^{p-1}G_n((K+C_n)\cap(r\theta+K+C_n))dr\right)^\frac{1}{p}\cr
&=&\left(\frac{G_n(K+C_n)}{G_n(K)}\right)^\frac{1}{p}\rho_{K_p(\tilde{g}_{K+C_n})}(\theta).
\end{eqnarray*}
Therefore, $\tilde{K}_p(\tilde{g}_{K+C_n})$ is the dilation of $K_p(\tilde{g}_{K+C_n})$ in the right-hand side of \eqref{eq:inclusionBallBodiesDiscreteCovariogram}:
$$
\tilde{K}_p(\tilde{g}_{K+C_n})=\left(\frac{G_n(K+C_n)}{G_n(K)}\right)^{\frac{1}{p}}K_p(\tilde{g}_{K+C_n}).
$$
Let $x,y\in K_p(\tilde{g}_K)$, $\lambda,\mu\in[0,1]$ such that $\lambda+\mu=1$, and $\gamma=\frac{1}{p}$. Let us define the functions
\begin{itemize}
\item $h(t)=\tilde{g}_{K+C_n}(t^\gamma(\lambda x+\mu y))$, $t> 0$,
\item $w(r)=\tilde{g}_K(r^\gamma x)$, $r> 0$,
\item $g(s)=\tilde{g}_K(s^\gamma y)$, $s>0$.
\end{itemize}
Denoting by $M_{-\gamma}^\lambda(r,s)$ for any $r,s>0$ the number
$$
M_{-\gamma}^\lambda(r,s)=(\lambda r^{-\gamma}+\mu s^{-\gamma})^{-\frac{1}{\gamma}},
$$
our purpose is to show that for any $r,s>0$ we have that
\begin{equation}\label{eq:inequalityMeans}
h(M_{-\gamma}^\lambda(r,s))\geq w(r)^{\frac{\lambda s^\gamma}{\lambda s^\gamma+\mu r^\gamma}}g(s)^{\frac{\mu r^\gamma}{\lambda s^\gamma+\mu r^\gamma}}
\end{equation}
in order to apply \cite[Theorem 2.5.4]{BGVV} and obtain
\begin{equation}\label{eq:InequalityIntegrals}
\int_0^\infty h(t)dt\geq M_{-\gamma}^\lambda\left(\int_0^\infty w(r)dr,\int_0^\infty g(s)ds\right).
\end{equation}

If $K\cap(r^\gamma x+K)=\emptyset$ or $K\cap(s^\gamma y+K)=\emptyset$ then $w(r)=0$ or $g(s)=0$ and inequality \eqref{eq:inequalityMeans} is trivial. Otherwise, calling
\begin{itemize}
\item $\lambda_1=\frac{\lambda s^\gamma}{\lambda s^\gamma+\mu r^\gamma}$
\item $\mu_1=\frac{\mu r^\gamma}{\lambda s^\gamma+\mu r^\gamma}$
\end{itemize}
we have that $\lambda_1+\mu_1=1$ and, since $K$ is convex,
\begin{eqnarray*}
&&(K+C_n)\cap(\lambda_1r^\gamma x+\mu_1s^\gamma y+K+C_n)\supseteq\cr
&\supseteq& \lambda_1(K\cap(r^\gamma x+K))+\mu_1(K\cap(s^\gamma y+K))+C_n
\end{eqnarray*}
and then, by the discrete Brunn-Minkowski inequality (Theorem \ref{thm: BM_lattice_point_no_G(K)G(L)>0}),
\begin{eqnarray*}
&&G_n((K+C_n)\cap(\lambda_1r^\gamma x+\mu_1s^\gamma y+K+C_n))\geq\cr
&\geq& G_n(\lambda_1(K\cap(r^\gamma x+K))+\mu_1(K\cap(s^\gamma y+K))+C_n)\cr
&\geq&\left(\lambda_1G_n^\frac{1}{n}(K\cap(r^\gamma x+K))+\mu_1G_n^\frac{1}{n}(K\cap(s^\gamma y+K))\right)^n\cr
&\geq&G_n^{\lambda_1}(K\cap(r^\gamma x+K))G_n^{\mu_1}(K\cap(s^\gamma y+K)).
\end{eqnarray*}
Therefore,
\begin{eqnarray*}
h(M_{-\gamma}^\lambda(r,s))&=&\tilde{g}_{K+C_n}\left(\frac{1}{\lambda r^{-\gamma}+\mu s^{-\gamma}}(\lambda x+\mu y)\right)\cr
&=&\tilde{g}_{K+C_n}\left(\frac{\lambda s^\gamma}{\lambda s^\gamma+\mu r^\gamma}r^\gamma x+\frac{\mu r^\gamma}{\lambda s^\gamma+\mu r^\gamma}s^\gamma y\right)\cr
&=&\tilde{g}_{K+C_n}\left(\lambda_1r^\gamma x+\mu_1s^\gamma y\right)\cr
&\geq&\tilde{g}_K^{\lambda_1}(r^\gamma x)\tilde{g}_K^{\mu_1}(s^\gamma y)\cr
&=&\tilde{g}_K(r^\gamma x)^{\frac{\lambda s^\gamma}{\lambda s^\gamma+\mu r^\gamma}}\tilde{g}_K(s^\gamma y)^{\frac{\mu r^\gamma}{\lambda s^\gamma+\mu r^\gamma}}\cr
&=&w(r)^{\frac{\lambda s^\gamma}{\lambda s^\gamma+\mu r^\gamma}}g(s)^{\frac{\mu r^\gamma}{\lambda s^\gamma+\mu r^\gamma}}.\cr
\end{eqnarray*}

Thus, by \cite[Theorem 2.5.4]{BGVV} we have \eqref{eq:InequalityIntegrals}. Equivalently, taking into account that $-\gamma<0$,
\begin{eqnarray*}
\left(\int_0^\infty \tilde{g}_{K+C_n}(t^\gamma(\lambda x+\mu y))dt\right)^{-\gamma}&\leq& \lambda\left(\int_0^\infty \tilde{g}_K(r^\gamma x)dr\right)^{-\gamma}\cr
&+&\mu\left(\int_0^\infty \tilde{g}_K(s^\gamma y)ds\right)^{-\gamma}.
\end{eqnarray*}
Changing variables, using that $\gamma=\frac{1}{p}$, and taking into account that $x,y\in K_p(\tilde{g}_K)$,
\begin{eqnarray*}
\left(\int_0^\infty pt^{p-1}\tilde{g}_{K+C_n}(t(\lambda x+\mu y))dt\right)^{-\frac{1}{p}}&\leq& \lambda\left(\int_0^\infty pr^{p-1}\tilde{g}_K(r x)dr\right)^{-\frac{1}{p}}\cr
&+&\mu\left(\int_0^\infty ps^{p-1}\tilde{g}_K(s y)ds\right)^{-\frac{1}{p}}\cr
&\leq&\lambda G_n(K)^{-\frac{1}{p}}+\mu G_n(K)^{-\frac{1}{p}}=G_n(K)^{-\frac{1}{p}}.
\end{eqnarray*}
Therefore
$$
\int_0^\infty pt^{p-1}\tilde{g}_{K+C_n}(t(\lambda x+\mu y))dt\geq G_n(K)
$$
and
$$
\lambda x+\mu y\in \tilde{K}_p(\tilde{g}_{K+C_n})=\left(\frac{G_n(K+C_n)}{G_n(K)}\right)^\frac{1}{p}K_p(\tilde{g}_{K+C_n}).
$$
\end{proof}

\begin{rmk}
Let us point out that a reverse inclusion of $K_p(\tilde{g}_{K+C_n})$ in a dilation of the convex hull of $\tilde{K}_p(\tilde{g}_K)$ is not possible as the following example shows: Consider $K=\textrm{conv}\{(0,0),\left(\frac{1}{2},1\right),\left(\frac{1}{2},-1\right)\}\subseteq\R^2$. Since $K\cap\Z^2=\{(0,0)\}$, by Remark \ref{rmk:OnePoint}, for any $p>0$ we have that $K_p(\tilde{g}_K)=-K$
and then
$$
\textrm{conv}(K_p(\tilde{g}_K))=-K\subseteq \{x\in\R^2\,:\,\langle x,e_2\rangle\leq0\}.
$$
Therefore, if $K_p(\tilde{g}_{K+C_2})$ is contained in a dilation of $\textrm{conv}(K_p(\tilde{g}_K))$, necessarily $\rho_{K_p(\tilde{g}_{K+C_2})}(e_2)=0$. However,
$$
(0,0)\in (K+C_2)\cap (re_2+K+C_2)\quad\textrm{for every }r\in[0,1)
$$
and then $\tilde{g}_{K+C_n}(re_2)>0$ for every $r\in[0,1)$. Thus, $\rho_{K_p(\tilde{g}_{K+C_2})}(e_2)>0$ and $K_p(\tilde{g}_{K+C_2})$ is not contained in any dilation of $\textrm{conv}(K_p(\tilde{g}_K))$.
\end{rmk}

The following lemma shows that even though it is defined from the discrete covariogram, the volume of $K_n(\tilde{g}_K)$ equals the volume of $K$.

\begin{lemma}\label{lem:Volume K_n(gK)}
Let $K\subseteq\R^n$ be a convex body with $0\in K$. Let $h:\R^n\to[0,\infty)$ be a homogeneous function of degree $p\geq0$. Then
$$
\int_{K_{n+p}(\tilde{g}_K)}h(x)dx=\frac{1}{G_n(K)}\int_{\R^n}h(x)G_n(K\cap(x+K))dx.
$$
In particular,
$$
\vol_n(K_n(\tilde{g}_K))=\vol_n(K).
$$
\end{lemma}
\begin{proof}
Integrating in polar coordinates we have that
\begin{eqnarray*}
&&\int_{K_{n+p}(\tilde{g}_K)}h(x)dx=n\vol_n(B_2^n)\int_{S^{n-1}}\int_0^{\rho_{K_{n+p}(\tilde{g}_K)}(\theta)}r^{n-1}h(r\theta)drd\sigma(\theta)\cr
&=&n\vol_n(B_2^n)\int_{S^{n-1}}\int_0^{\rho_{K_{n+p}(\tilde{g}_K)}(\theta)}r^{n+p-1}h(\theta)drd\sigma(\theta)\cr
&=&\frac{n}{n+p}\vol_n(B_2^n)\int_{S^{n-1}}\rho_{K_{n+p}(\tilde{g}_K)}^{n+p}(\theta)h(\theta)d\sigma(\theta)\cr
&=&\frac{n\vol_n(B_2^n)}{G_n(K)}\int_{S^{n-1}}\int_0^\infty r^{n+p-1}G_n(K\cap(r\theta+K))h(\theta)drd\sigma(\theta)\cr
&=&\frac{n\vol_n(B_2^n)}{G_n(K)}\int_{S^{n-1}}\int_0^\infty r^{n-1}G_n(K\cap(r\theta+K))h(r\theta)drd\sigma(\theta)\cr
&=&\frac{1}{G_n(K)}\int_{\R^n}h(x)G_n(K\cap(x+K))dx.
\end{eqnarray*}

Notice that if $h(x)=1$, which is homogeneous of degree $0$, we have that
\begin{eqnarray*}
\vol_n(K_n(\tilde{g}_K))&=&\frac{1}{G_n(K)}\int_{\R^n}G_n(K\cap(x+K))dx\cr
&=&\frac{1}{G_n(K)}\int_{\R^n}\sum_{y\in \Z^n}\chi_{K}(y)\chi_{x+K}(y)dx\cr
&=&\frac{1}{G_n(K)}\sum_{y\in \Z^n}\int_{\R^n}\chi_{K}(y)\chi_{y-K}(x)dx\cr
&=&\frac{1}{G_n(K)}\sum_{y\in \Z^n}\chi_{K}(y)\vol_n(K)=\vol_n(K).
\end{eqnarray*}
\end{proof}

Let us now prove the inclusion relation given in Theorem \ref{thm:InclusionBallBodiesDiscreteCovariogram}. The proof will follow the lines of the proof of Theorem \ref{thm:BerwaldDiscrete}. We begin with the following lemma in which, for any $p>0$ and any $\theta\in S^{n-1}$, we construct a $\frac{1}{n}$-affine function on its support, $g_{p,\theta}$, with the property that for any $q>0$
$$
\left(\frac{{{n+q}\choose{n}}}{G_n(K)}\int_0^\infty qr^{q-1}g_{p,\theta}(r)dr\right)^\frac{1}{q}=m_{p,\theta},
$$
where
$$
m_{p,\theta}:={{n+p}\choose{n}}^\frac{1}{p}\rho_{\tilde{K}_p(\tilde{g}_{K+C_n})}(\theta),
$$
being $\tilde{K}_p(\tilde{g}_{K+C_n})$ the dilation of $K_p(\tilde{g}_{K+C_n})$ defined in Proposition \ref{prop:ConvexityBallBodies}. Such function has a crossing point $r_0(p,\theta)$ (see Lemma \ref{lem:r_0} for the precise definition of such crossing point) with a modification of the function $\tilde{g}_K$.

\begin{lemma}\label{lem:r_0}
Let $K\subseteq\R^n$ be a convex body such that $0\in K$ and let $\theta\in S^{n-1}$. For any $p>0$, let
$$
m_{p,\theta}:=\left(\frac{{{n+p}\choose{n}}}{G_n(K)}\int_0^\infty pr^{p-1}\tilde{g}_{K+C_n}(r\theta)dr\right)^\frac{1}{p}
$$
and $g_{p,\theta}:[0,\infty)\to[0,\infty)$ be the function given by
$$
g_{p,\theta}(r)=\begin{cases}
\left(1-\frac{r}{m_{p,\theta}}\right)^nG_n(K)&\textrm{ if }0\leq r\leq m_{p,\theta}\\
0 &\textrm{ otherwise}.
\end{cases}
$$
Then, there exists $r_0(p,\theta)\in[0,\infty)$ such that
$$
\tilde{g}_{K+C_n}(r\theta)\geq g_{p,\theta}(r)
$$
for every $0\leq r<r_0(p,\theta)$ and
$$
g_{p,\theta}(r)\geq \tilde{g}_K(r\theta)
$$
for every $r\geq r_0(p,\theta)$.
\end{lemma}

\begin{proof}
Let $K\subseteq\R^n$ be a convex body with $0\in K$, $\theta\in S^{n-1}$ and $p>0$. Let $m_{p,\theta}$ and $g_{p,\theta}$ be defined as in the statement. First of all, notice that since $0\in K$ we have that $0\in\textrm{int}(K+C_n)$ and then $m_{p,\theta}>0$. Notice also that for every convex set $L$, every $\theta\in S^{n-1}$, and every $0\leq r_1<r_2$ we have that $L\cap(r_1\theta +L)\supseteq L\cap(r_2\theta+L)$ and then the functions $\tilde{g}_K(r\theta)$ and $\tilde{g}_{K+C_n}(r\theta)$ are decreasing in $r\in[0,\infty)$. Moreover, since $K\subseteq K+C_n$ we have that $K\cap(x+K)\subseteq (K+C_n)\cap(x+K+C_n)$ for every $x\in\R^n$ and then $\tilde{g}_K(x)\leq \tilde{g}_{K+C_n}(x)$ for every $x\in\R^n$. Furthermore, since $K$ is a compact convex set, $\tilde{g}_K(r\theta)$ is continuous from the left in $r\in[0,\infty)$ and since $K+C_n$ is an open bounded convex set, $\tilde{g}_{K+C_n}(r\theta)$ is continuous from the right in $r\in[0,\infty)$. Let us call
\begin{equation}\label{eq:DefinitionMtheta}
M_\theta=\max\{r\geq0\,:\, \tilde{g}_{K}(r\theta)\geq 1\}
\end{equation}
and
$$
\tilde{M}_\theta=\sup\{r\geq0\,:\, \tilde{g}_{K+C_n}(r\theta)\geq 1\}
$$
and notice that, necessarily, $m_{p,\theta}\geq M_\theta$. Otherwise, if $m_{p,\theta}<M_\theta\leq\tilde{M}_\theta$ then \\$K\cap(m_{p,\theta}\theta+K)\neq\emptyset$ and for every $0\leq r\leq m_{p,\theta}$ we have that
$$
(K+C_n)\cap\left(r\theta+K+C_n\right)\supseteq\left(1-\frac{r}{m_{p,\theta}}\right)K+\frac{r}{m_{p,\theta}}(K\cap(m_{p,\theta}\theta+K))+C_n
$$
and then, by the discrete Brunn-Minkowski inequality (Theorem \ref{thm: BM_lattice_point_no_G(K)G(L)>0}),
\begin{eqnarray*}
\tilde{g}_{K+C_n}^\frac{1}{n}(r\theta)&\geq & G_n\left(\left(1-\frac{r}{m_{p,\theta}}\right)K+\frac{r}{m_{p,\theta}}(K\cap(m_{p,\theta}\theta+K))+C_n\right)^\frac{1}{n}\cr &\geq&\left(1-\frac{r}{m_{p,\theta}}\right)G_n(K)^\frac{1}{n}+\frac{r}{m_{p,\theta}}G_n(K\cap(m_{p,\theta}\theta+K))^\frac{1}{n}\cr
&\geq&\left(1-\frac{r}{m_{p,\theta}}\right)G_n(K)^\frac{1}{n}.
\end{eqnarray*}
Therefore, for every $0\leq r\leq m_{p,\theta}$ we have that
$$
\tilde{g}_{K+C_n}(r\theta)\geq g_{p,\theta}(r)
$$
and then, taking into account that $\tilde{g}_{K+C_n}(r\theta)\geq 1$ for every $0\leq r<\tilde{M}_\theta$
\begin{eqnarray*}
m_{p,\theta}^p&=&\frac{{{n+p}\choose{n}}}{G_n(K)}\int_0^\infty pr^{p-1}\tilde{g}_{K+C_n}(r\theta)dr\cr
&=&\frac{{{n+p}\choose{n}}}{G_n(K)}\int_0^{\tilde{M}_\theta} pr^{p-1}\tilde{g}_{K+C_n}(r\theta)dr\cr
&>&\frac{{{n+p}\choose{n}}}{G_n(K)}\int_0^{m_{p,\theta}} pr^{p-1}\tilde{g}_{K+C_n}(r\theta)dr\cr
&\geq&\frac{{{n+p}\choose{n}}}{G_n(K)}\int_0^{m_{p,\theta}} pr^{p-1}g_{p,\theta}(r)dr\cr
&=&{{n+p}\choose{n}}\int_0^{m_{p,\theta}} pr^{p-1}\left(1-\frac{r}{m_{p,\theta}}\right)dr=m_{p,\theta}^p,\cr
\end{eqnarray*}
which is a contradiction.

Since we trivially have that for any $r>\tilde{M}_\theta$
$$
0=\tilde{g}_{K+C_n}(r\theta)\leq g_{p,\theta}(r),
$$
we can define
$$
r_0(p,\theta)=\inf\{r\geq0\,:\tilde{g}_{K+C_n}(r\theta)\leq g_{p,\theta}(r)\}<\infty.
$$
From the definition of $r_0(p,\theta)$ as an infimum we trivially have that for every $0\leq r<r_0(p,\theta)$
$$
\tilde{g}_{K+C_n}(r\theta)> g_{p,\theta}(r).
$$
Besides, since $\tilde{g}_{K+C_n}(r\theta)$ is continuous from the right on $[0,\infty)$ and $\tilde{g}_K(x)\leq\tilde{g}_{K+C_n}(x)$ for every $x\in\R^n$, from the definition of $r_0(p,\theta)$ as an infimum we obtain
\begin{equation}\label{eq:ComparisonAtr_0}
\tilde{g}_{K}(r_0(p,\theta)\theta)\leq\tilde{g}_{K+C_n}(r_0(p,\theta)\theta)\leq g_{p,\theta}(r_0(p,\theta)).
\end{equation}
and $r_0(p,\theta)$ is a minimum.

Notice that also, from the definition of $M_\theta$, for every $r>M_\theta$ we have that
$$
g_{p,\theta}(r)\geq \tilde{g}_K(r\theta)=0.
$$
Therefore, if $r_0(p,\theta)> M_\theta$ the theorem is proved. Thus, we will assume that $r_0(p,\theta)\leq M_\theta$.

If $r_0(p,\theta)= M_\theta$ then, by \eqref{eq:ComparisonAtr_0} we have that
$$
\tilde{g}_{K}(M_\theta\theta)\leq\tilde{g}_{K+C_n}(M_\theta\theta)\leq g_{p,\theta}(M_\theta)
$$
and the theorem is proved.

Let us assume, then, that $r_0(p,\theta)<M_\theta$ and let us prove that if $r_0(p,\theta)< r\leq M_\theta$ then $g_{p,\theta}(r)\geq \tilde{g}_K(r\theta)$.

If $r_0(p,\theta)<r\leq M_\theta$ then, taking $\lambda=\frac{r_0(p,\theta)}{r}\in[0,1)$ we have
$$
(K+C_n)\cap(r_0(p,\theta)\theta+K+C_n)\supseteq (1-\lambda) K+\lambda (K\cap(r\theta+K))+C_n
$$
and then by the discrete Brunn-Minkowski inequality (Theorem \ref{thm: BM_lattice_point_no_G(K)G(L)>0})
\begin{eqnarray*}
\tilde{g}_{K+C_n}^\frac{1}{n}(r_0(p,\theta)\theta)&\geq& G_n\left((1-\lambda) K+\lambda (K\cap(r\theta+K)+C_n)\right)^\frac{1}{n}\cr
&\geq&(1-\lambda)G_n(K)^\frac{1}{n}+\lambda G_n(K\cap(r\theta+K))^\frac{1}{n}\cr
&=&(1-\lambda)\tilde{g}_K^\frac{1}{n}(0)+\lambda \tilde{g}_K^\frac{1}{n}(r\theta).\cr
\end{eqnarray*}

Taking into account that $r_0(p,\theta)<r\leq M_\theta\leq m_{p,\theta}$
\begin{eqnarray*}
g_{p,\theta}^\frac{1}{n}(r_0(p,\theta))&=&\left(1-\frac{r_0(p,\theta)}{m_{p,\theta}}\right)G_n(K)^\frac{1}{n}\cr
&=&(1-\lambda)G_n(K)^\frac{1}{n}+\lambda\left(1-\frac{r}{m_{p,\theta}}\right)G_n(K)^\frac{1}{n}\cr
&=&(1-\lambda)\tilde{g}_K^\frac{1}{n}(0)+\lambda g_{p,\theta}^\frac{1}{n}(r).
\end{eqnarray*}
Thus, by \eqref{eq:ComparisonAtr_0}, if $r_0(p,\theta)<r\leq M_\theta$
\begin{eqnarray*}
(1-\lambda)\tilde{g}_K^\frac{1}{n}(0)+\lambda g_{p,\theta}^\frac{1}{n}(r)&=&g_{p,\theta}^\frac{1}{n}(r_0(p,\theta))\geq\tilde{g}_{K+C_n}^\frac{1}{n}(r_0(p,\theta)\theta)\cr
&\geq&(1-\lambda)\tilde{g}_K^\frac{1}{n}(0)+\lambda \tilde{g}_K^\frac{1}{n}(r\theta)
\end{eqnarray*}
and then
$$
 g_{p,\theta}(r)\geq \tilde{g}_K(r\theta).
$$
\end{proof}

\begin{rmk}\label{rmk:RadialFunctionKcapZn-K}
Let us point out that for any $\theta\in S^{n-1}$ we have that $\tilde{g}_{K}(r\theta)\geq 1$ if and only if $(K\cap\Z^n)\cap(r\theta+K)=K\cap(r\theta+K)\cap\Z^n\neq\emptyset$, which happens if and only if $r\theta\in (K\cap\Z^n)-K$. Therefore, since $\tilde{g}_{K}(r\theta)$ is decreasing in $r\in[0,\infty)$ we have that $(K\cap \Z^n)-K$ is a star set with $0$ as a center and the number $M_\theta$ defined in \eqref{eq:DefinitionMtheta} is $M_\theta=\rho_{(K\cap \Z^n)-K}(\theta)$. In the same way, $((K+C_n)\cap \Z^n)-(K+C_n)$ is a star set with $0$ as a center and $\tilde{M}_\theta=\rho_{((K+C_n)\cap \Z^n)-(K+C_n)}(\theta)$
\end{rmk}
Now, we can prove the inclusion relation given by Theorem \ref{thm:InclusionBallBodiesDiscreteCovariogram}.

\begin{proof}[Proof of Theorem \ref{thm:InclusionBallBodiesDiscreteCovariogram}]
Let $K\subseteq\R^n$ be a convex body with $0\in K$. Let $\theta\in S^{n-1}$. Fix $p>0$ and let, as in Lemma \ref{lem:r_0}, $m_{p,\theta}$ be defined as
$$
m_{p,\theta}=\left(\frac{{{n+p}\choose{n}}}{G_n(K)}\int_0^\infty pr^{p-1}\tilde{g}_{K+C_n}(r\theta)dr\right)^\frac{1}{p}
$$
and $g_{p,\theta}:[0,\infty)\to[0,\infty)$ be the function given by
$$
g_{p,\theta}(r)=\begin{cases}
\left(1-\frac{r}{m_{p,\theta}}\right)^nG_n(K)&\textrm{ if }0\leq r\leq m_{p,\theta}\\
0 &\textrm{ otherwise}.
\end{cases}
$$
Notice that, changing variables $r=m_{p,\theta}s$, for every $q>0$ we have that
\begin{eqnarray*}
\frac{{{n+q}\choose{n}}}{G_n(K)}\int_0^\infty qr^{q-1}g_{p,\theta}(r)dr&=&\frac{{{n+q}\choose{n}}}{G_n(K)}\int_0^{m_{p,\theta}} qr^{q-1}\left(1-\frac{r}{m_{p,\theta}}\right)^nG_n(K)dr\cr
&=&{{n+q}\choose{n}}m_{p,\theta}^q\int_0^1 qs^{q-1}(1-s)^nds\cr
&=&m_{p,\theta}^q.
\end{eqnarray*}
In particular, taking $q=p$, we have that
$$
\int_0^\infty r^{p-1}g_{p,\theta}(r)dr=\int_0^\infty r^{p-1}\tilde{g}_{K+C_n}(r\theta)dr.
$$

Let $q>p$. Taking $r_0(p,\theta)$ provided by Lemma \ref{lem:r_0}, we have
\begin{eqnarray*}
&&\int_0^{r_0(p,\theta)}r^{q-1}\left(\tilde{g}_{K+C_n}(r\theta)-g_{p,\theta}(r)\right)dr-\int_{r_0(p,\theta)}^\infty r^{q-1}\left(g_{p,\theta}(r)-\tilde{g}_K(r\theta)\right)dr\cr
&=&\int_0^{r_0(p,\theta)}r^{q-p}r^{p-1}\left(\tilde{g}_{K+C_n}(r\theta)-g_{p,\theta}(r)\right)dr\cr
&-&\int_{r_0(p,\theta)}^\infty r^{q-p}r^{p-1}\left(g_{p,\theta}(r)-\tilde{g}_K(r\theta)\right)dr\cr
&\leq&r_0(p,\theta)^{q-p}\int_0^{r_0(p,\theta)}r^{p-1}\left(\tilde{g}_{K+C_n}(r\theta)-g_{p,\theta}(r)\right)dr\cr
&-&r_0(p,\theta)^{q-p}\int_{r_0}^\infty r^{p-1}\left(g_{p,\theta}(r)-\tilde{g}_K(r\theta)\right)dr\cr
&=&r_0(p,\theta)^{q-p}\int_0^{r_0(p,\theta)}r^{p-1}\left(\tilde{g}_{K+C_n}(r\theta)-g_{p,\theta}(r)\right)dr\cr
&+&r_0(p,\theta)^{q-p}\int_{r_0(p,\theta)}^\infty r^{p-1}\left(\tilde{g}_K(r\theta)-g_{p,\theta}(r)\right)dr\cr
&\leq&r_0(p,\theta)^{q-p}\int_0^{r_0(p,\theta)}r^{p-1}\left(\tilde{g}_{K+C_n}(r\theta)-g_{p,\theta}(r)\right)dr\cr
&+&r_0(p,\theta)^{q-p}\int_{r_0(p,\theta)}^\infty r^{p-1}\left(\tilde{g}_{K+C_n}(r\theta)-g_{p,\theta}(r)\right)dr\cr
&=&r_0(p,\theta)^{q-p}\int_0^{\infty}r^{p-1}\left(\tilde{g}_{K+C_n}(r\theta)-g_{p,\theta}(r)\right)dr=0.
\end{eqnarray*}

Therefore, we have that
$$
\int_0^{r_0(p,\theta)}r^{q-1}\tilde{g}_{K+C_n}(r\theta)dr+\int_{r_0(p,\theta)}^\infty r^{q-1}\tilde{g}_{K}(r\theta)dr\leq\int_0^\infty r^{q-1}g_{p,\theta}(r)dr
$$
and then, since $\tilde{g}_K(x)\leq\tilde{g}_{K+C_n}(x)$ for every $x\in\R^n$,
$$
\int_{0}^\infty r^{q-1}\tilde{g}_{K}(r\theta)dr\leq\int_0^\infty r^{q-1}g_{p,\theta}(r)dr.
$$
Consequently,
\begin{eqnarray*}
{{n+q}\choose{n}}^\frac{1}{q}\rho_{K_q(\tilde{g}_K)}(\theta)&=&\left(\frac{{{n+q}\choose{n}}}{G_n(K)}\int_0^\infty qr^{q-1}\tilde{g}_{K}(r\theta)dr\right)^\frac{1}{q}\cr
&\leq&\left(\frac{{{n+q}\choose{n}}}{G_n(K)}\int_0^\infty qr^{q-1}g_{p,\theta}(r)dr\right)^\frac{1}{q}=m_{p,\theta}\cr
&=&\left(\frac{{{n+p}\choose{n}}}{G_n(K)}\int_0^\infty pr^{p-1}\tilde{g}_{K+C_n}(r\theta)dr\right)^\frac{1}{p}\cr
&=&\left(\frac{{{n+p}\choose{n}}G_n(K+C_n)}{G_n(K)}\right)^\frac{1}{p}\rho_{K_p(\tilde{g}_{K+C_n})}(\theta).
\end{eqnarray*}
Since this is true for every $\theta\in S^{n-1}$ we have the inclusion relation stated in the theorem.
\end{proof}

Finally, let us point out that, as it was shown in \cite{GZ}, the inclusion relation given by \eqref{eq:Inclusionpq} provides the following inclusion relation for the difference body $K-K$ by making $q\to\infty$:
$$
K-K\subseteq{{n+p}\choose{n}}^\frac{1}{p}R_p(K),\quad\textrm{ for all } p>-1
$$
and, since $R_p(K)=K_p(g_K)$ for every $p>0$,
$$
K-K\subseteq{{n+p}\choose{n}}^\frac{1}{p}K_p(g_K),\quad\textrm{ for all } p>0.
$$
Let us show that as a corollary, making $q\to\infty$ in Theorem \ref{thm:InclusionBallBodiesDiscreteCovariogram}, we can obtain an inclusion relation for the set
$$
(K\cap\Z^n)-K=\bigcup_{x\in K\cap\Z^n}(x-K)
$$
which is a slightly smaller set than the difference body
$$
K-K=\bigcup_{x\in K}(x-K).
$$
\begin{cor}\label{cor:DifferenceBody}
Let $K\subseteq\R^n$ be a convex body such that $0\in K$. For any $p>0$ we have that
$$
(K\cap\Z^n)-K\subseteq {{n+p}\choose{n}}^\frac{1}{p}\left(\frac{G_n(K+C_n)}{G_n(K)}\right)^\frac{1}{p}K_p(\tilde{g}_{K+C_n}).
$$
\end{cor}
\begin{proof}
Let $K\subseteq\R^n$ be a convex body with $0\in K$, $\theta\in S^{n-1}$ and $p>0$. We have seen in Theorem \ref{thm:InclusionBallBodiesDiscreteCovariogram} that for any $0<p<q$ we have that
$$
{{n+q}\choose{n}}^\frac{1}{q}\rho_{K_q(\tilde{g}_K)}(\theta)\leq\left(\frac{{{n+p}\choose{n}}G_n(K+C_n)}{G_n(K)}\right)^\frac{1}{p}\rho_{K_p(\tilde{g}_{K+C_n})}(\theta).
$$
Taking limits as $q\to\infty$ we have that $\displaystyle{\lim_{q\to\infty}{{n+q}\choose{n}}^\frac{1}{q}=1}$. Besides, calling as in \eqref{eq:DefinitionMtheta}, $M_\theta=\max\{r\geq0\,:\, \tilde{g}_{K}(r\theta)\geq 1\}$, if we consider the measure $d\nu(r)=\tilde{g}_K(r\theta)dr$ on the interval $[0,M_\theta]$, denote by $L^q(\nu)$ the corresponding Lebesgue space on $([0,M_\theta],d \nu)$, and take $h(r)=r$, then
$$
M_\theta=\Vert h\Vert_\infty=\lim_{q\to\infty}\Vert h\Vert_{L^{q-1}(\nu)}=\lim_{q\to\infty}\left(\int_0^{M_\theta}r^{q-1}\tilde{g}_K(r\theta)dr\right)^\frac{1}{q-1}.
$$
Therefore,
\begin{eqnarray*}
\lim_{q\to\infty}\rho_{K_q(\tilde{g}_K)}&=&\lim_{q\to\infty}\left(\frac{q}{G_n(K)}\int_0^{\infty}r^{q-1}\tilde{g}_K(r\theta)dr\right)^\frac{1}{q}\cr
&=&\lim_{q\to\infty}\left(\frac{q}{G_n(K)}\int_0^{M_\theta}r^{q-1}\tilde{g}_K(r\theta)dr\right)^\frac{1}{q}\cr
&=&\lim_{q\to\infty}\frac{q^\frac{1}{q}}{G_n(K)^\frac{1}{q}}\left(\int_0^{M_\theta}r^{q-1}\tilde{g}_K(r\theta)dr\right)^{\frac{1}{q-1}\frac{q-1}{q}}\cr
&=&M_\theta.
\end{eqnarray*}
Taking into account Remark \ref{rmk:RadialFunctionKcapZn-K}, we have
$$
\rho_{(K\cap \Z^n)-K}(\theta)\leq\left(\frac{{{n+p}\choose{n}}G_n(K+C_n)}{G_n(K)}\right)^\frac{1}{p}\rho_{K_p(\tilde{g}_{K+C_n})}(\theta).
$$
Since this is true for every $\theta\in S^{n-1}$ we obtain the result.
\end{proof}

\end{document}